\documentclass[12pt,twoside,reqno]{amsart}
\usepackage[utf8]{inputenc}
\usepackage[T1]{fontenc}
\usepackage[english]{babel}
\usepackage[colorlinks=true,citecolor=blue]{hyperref}
\hypersetup{
  hidelinks
}
\usepackage{mathptmx,amsmath,amssymb,amsfonts,amsthm,enumerate,color}
\usepackage{mathtools,booktabs,array}

\usepackage[most]{tcolorbox}
\usepackage{todonotes}

\usepackage{xcolor}

\newcommand{\R}{\mathbb R}
\newcommand{\Rp}{\mathbb R_+}
\newcommand{\cF}{\mathcal F}
\newcommand{\cG}{\mathcal G}
\newcommand{\cH}{\mathcal H}
\newcommand{\cK}{\mathcal K}
\newcommand{\cP}{\mathcal P}
\newcommand{\cM}{\mathcal M}
\newcommand{\AX}{\mathrm{ax}}
\newcommand{\nc}{\mathrm{nc}}
\DeclareMathOperator{\intt}{int}
\DeclareMathOperator{\cl}{cl}
\DeclareMathOperator{\cone}{cone}
\DeclareMathOperator{\co}{co}
\DeclareMathOperator{\dom}{dom}
\DeclareMathOperator{\dist}{dist}
\DeclareMathOperator{\sdw}{sdw}
\DeclareMathOperator{\rad}{rad}
\DeclareMathOperator{\Lip}{Lip}
\DeclareMathOperator{\argmin}{argmin}

\newcommand{\pos}[1]{[#1]_+}

\newtheorem{theorem}{Theorem}[section]
\newtheorem{corollary}{Corollary}[section]
\newtheorem{lemma}{Lemma}[section]
\newtheorem{proposition}{Proposition}[section]
\theoremstyle{definition}
\newtheorem{definition}{Definition}[section]
\newtheorem{example}{Example}[section]
\newtheorem{remark}{Remark}[section]

\numberwithin{equation}{section}

\begin{document}
\setcounter{page}{1}

\vspace*{1.0cm}

\title[Directional Fineness for Radiant and Co-Radiant Sets]
{Directional Fineness for Radiant and Co-Radiant Sets:
Separation and Quantitative Approximation}

\author[F. Garc\'ia-Casta\~no, M.\,\'A. Melguizo-Padial]
{Fernando Garc\'ia-Casta\~no$^{1,*}$, Miguel \'Angel Melguizo-Padial$^{1}$}
\maketitle
\vspace*{-0.6cm}

\begin{center}
{\footnotesize\it
$^1$Department of Mathematics, University of Alicante, Campus de Sant Vicent del Raspeig s/n,
03690 Sant Vicent del Raspeig (Alicante), Spain
}
\end{center}

\vskip 4mm
{\small\noindent {\bf Abstract.}
A prescribed directionally fine family of cones yields separation and primal
bipolar identities for closed radiant and co-radiant sets in real normed spaces.
Positive levels of adapted
positively homogeneous
functions separate
points from closed radiant sets, while a correction by the norm
provides bounded radiant neighbourhoods for co-radiant separation.
The resulting classes also calibrate radial thresholds and represent
gauges and cogauges through strict separators.  
For convex directionally fine families,
the separating generators can be chosen
Lipschitz and superlinear, yielding convex
separating superlevels.

Closed co-radiant sets with compact
direction sets admit finite relative outer
approximations by positive superlevels of
selected generators, with separation preservation
and directional and bounded-section error
estimates, without requiring continuity or
boundedness of the radial profile.
Superlinear generators
yield a global residual error bound, exact penalization for Lipschitz
objectives,
recovery estimates, and convergence of bounded approximate minimizers.

An explicit convex directionally fine axial
family is available in every real normed space.
Its norm-corrected functions realize separation
and bipolarity without rotundity assumptions,
while an axial recovery map improves the relative
distance factor from $\eta/(1-\eta)$ to $\eta$.
A three-ray production model
illustrates the quantitative recovery estimates.


}

\renewcommand{\thefootnote}{}
\footnotetext{$^*$Corresponding author.
\par
E-mail addresses: fernando.gc@ua.es (F. Garc\'ia-Casta\~no),
ma.mp@ua.es (M.\,\'A. Melguizo-Padial).

}

\section{Introduction}\label{sec:introduction}

A radiant set is invariant under contractions towards the origin; a
co-radiant set is invariant under dilations away from it. Separation and
approximation of these sets require control of both directions and radial
thresholds.  
Directional fineness, introduced for
families of cones in \cite{GarciaMelguizo2026},
ensures that, around every unit direction and
within every prescribed aperture, the family
contains a cone having that direction in its
interior. This localization provides the
geometric basis for the separating constructions
developed below.

Radiant sets are precisely sets star-shaped with respect to the origin and therefore belong to the broader framework of starshaped analysis. Their radial geometry is naturally encoded by Minkowski-type gauges, while the corresponding co-radiant geometry leads to cogauges and reverse polar constructions. In this setting, nonlinear separation plays a role analogous to that of hyperplane separation in classical convex analysis: it provides scalar descriptions of sets, polarity and bipolar representations, and support constructions for positively homogeneous functions. These connections have been developed in the framework of abstract convexity and starshaped analysis, in particular through superlinear separation, gauge and cogauge representations, and polarity relations \cite{Zaffaroni2004,Zaffaroni2006,Zaffaroni2007,Zaffaroni2008,Mohebi2013,Zaffaroni2013,HansenEtAl2020}. Thus separation is useful here not only as a geometric device for distinguishing a point from a set, but also as a mechanism for converting radial set geometry into scalar and dual representations.

Co-radiant sets were introduced into approximate vector optimization to
provide unified notions of $\varepsilon$-efficiency, first in
finite-dimensional multiobjective programming and subsequently in a
general nonconvex framework with nonlinear scalarizations
\cite{GutierrezJimenezNovo2006MMOR,GutierrezJimenezNovo2006SIAM}.
They have since been used in Ekeland-type variational principles in
vector optimization, scalarization methods, and variable ordering
structures
\cite{Qiu2013,SayadiBanderEtAl2017,
SayadiBanderKasimbeyliPourkarimi2017,PengWeiKasimbeyli2025}.
More recently, co-radiant sets have also appeared in set optimization,
through generalized oriented-distance scalarizations and optimality
conditions for quasi-minimal solutions
\cite{AnsariSharma2022}.
Related co-radiant function classes also arise in mathematical economics,
where increasing quasiconcave co-radiant functions are used in production
and utility theory, with co-radiance providing a mathematical formulation
of diminishing returns to scale
\cite{MartinezLegazRubinovSchaible2005}.
These developments motivate the following questions about co-radiant
sets: when such a set represents a domination or feasibility structure,
finite approximations, separation preservation, and explicit recovery
estimates become relevant in addition to qualitative representation.

Continuous superlinear point separation of closed radiant and co-radiant
sets in normed spaces, and their characterization as bipolar fixed points,
are known \cite{Zaffaroni2007}. Zaffaroni's normlinear realization uses a
locally uniformly rotund (LUR) norm \cite{Zaffaroni2022}. Our contribution
concerns the sufficiency of prescribed generating families and quantitative
approximation and recovery, rather than the existence of separation or
bipolarity alone.

We fix adapted positively homogeneous generators of a directionally fine
family independently of subsequent sets and points. Their positive
superlevels separate points from closed radiant sets. A correction by the
norm supplies bounded radiant neighbourhoods for the opposite separation.
Convex directional families admit Lipschitz superlinear generators.
The resulting evaluation polarities recover the closed radial hull and
the closed dilation shadow. Radial calibration and gauge representations
are consequences within these prescribed classes.

The main quantitative development concerns closed co-radiant sets \(D\)
with compact direction set. Finitely many positive superlevels yield an
outer approximation \(D^\eta\) with directional and bounded-section error
estimates, preserving separation from a disjoint closed radiant set.
Neither continuity nor boundedness of the radial profile is required.
For superlinear generators, an explicit directional repair proves a
global residual error bound for the finite union. Exact penalization of
a Lipschitz objective then follows by the usual Lipschitz comparison.
Recovery in \(D\) separately controls geometric error, feasible objective
values, and relaxed optimal values. Uniformly bounded approximate
minimizers converge subsequentially to minimizers of the original problem.

The axial family, introduced at the end of Section~\ref{sec:generators},
gives closed-form generators and separators in every real normed space.
This removes the LUR requirement by changing the separating class, not
by extending the normlinear class itself. Its explicit recovery map
improves the relative distance constant from \(\eta/(1-\eta)\) to \(\eta\).
Once norming functionals are chosen, the local axial constructions are
explicit and the parameter estimates guide their use. Compactness supplies
the finite global selection, but no certified selection procedure or bound
on the number of pieces. A three-ray production model illustrates the
recovery and value bounds.

Sections~\ref{sec:preliminaries} and~\ref{sec:generators} give the radial
and functional setting. Sections~\ref{sec:positive}--\ref{sec:polarity}
establish separation, radial representations, and bipolar identities.
Sections~\ref{sec:approximation} and~\ref{sec:optimization} treat finite
relative approximation, penalization, recovery, and convergence.
The axial specializations conclude the corresponding result sections,
and Section~\ref{sec:optimization} includes a production model.
Section~\ref{sec:comparison} distinguishes the earlier separation and
representation results from the present realizations and quantitative
estimates. Section~\ref{sec:conclusion} summarizes the conclusions.

\section{Radial geometry}\label{sec:preliminaries}
In this section, we fix the notation and basic terminology used throughout the paper. We also introduce the radial profiles of radiant and co-radiant sets and establish some fundamental properties of closed co-radiant sets, including those related to directional compactness, which will play an important role in the subsequent approximation and convergence results.

Let \(X\) be a nonzero real normed space, \(X^*\) its continuous dual,
\(B_X\) its closed unit ball, and \(S_X\) its unit sphere. Write
\(B(x,r)=x+rB_X\), \(0_X\) for the origin, \(\Rp=[0,+\infty)\),
\(\mathbb N=\{1,2,\ldots\}\), and \(\pos{t}=\max\{t,0\}\).
Unless stated otherwise, prescribed sets denoted by \(R\) and \(D\) are
nonempty and proper, with \(0_X\in R\) and \(0_X\notin D\).
The invariance conditions below also apply to hulls and level sets
that may be empty or equal to \(X\).
A cone is invariant
under nonnegative scalar multiplication; a proper cone is nontrivial and
distinct from \(X\). 
Convexity and pointedness are not implicit.
Throughout, every cone belonging to a prescribed family \(\cF\) is assumed to be proper and closed. This convention applies in particular to directionally fine families.
Operator norms are induced by the norm of \(X\).
For a subset \(A\) of a real vector space,
\(\cone(A)=\{ta:t\geq0,\ a\in A\}\), and \(\co(A)\) denotes the set of
finite convex combinations of elements of \(A\). Closures and interiors
are taken in the norm topology. For nonempty \(A,B\subset X\), set
\[
 \dist(x,A)=\inf_{a\in A}\|x-a\|,\qquad
 \dist(A,B)=\inf_{a\in A,\ b\in B}\|a-b\|,\qquad
 e(A,B)=\sup_{a\in A}\dist(a,B).
\]
The Hausdorff distance is \(d_H(A,B)=\max\{e(A,B),e(B,A)\}\);
in excess estimates we put \(e(\varnothing,B)=0\).
For a function \(h:X\to\R\cup\{+\infty\}\) and a nonempty set
\(A\subset X\), write \(\dom h=\{x\in X:h(x)<+\infty\}\) and
\(\argmin_A h=\{x\in A:h(x)=\inf_A h\}\). If \(h\) is real-valued,
we also write
\(\Lip(h)=\sup_{x\ne y}|h(x)-h(y)|/\|x-y\|\), possibly infinite.
The function \(h\) is \emph{Lipschitz} if \(\Lip(h)<+\infty\).
A family \((h_i)_{i\in I}\) of real-valued functions on \(X\) is
\emph{uniformly Lipschitz} if \(\sup_{i\in I}\Lip(h_i)<+\infty\).

\begin{definition}\label{def:radial}
A set \(R\) is \emph{radiant} if \(tR\subset R\) for \(0\leq t\leq1\);
a set \(D\) is \emph{co-radiant} if \(tD\subset D\) for \(t\geq1\).
The radial hull and dilation shadow of \(A\subset X\) are
\[
 \rad A=\{ta:a\in A,\ 0\leq t\leq1\},\qquad
 \sdw A=\{ta:a\in A,\ t\geq1\}.
\]
\end{definition}

Complements interchange the two classes. If \(D\) is closed and co-radiant
and \(y\notin D\), compactness of the segment gives
\begin{equation}\label{eq:segment-gap}
 [0_X,y]\cap D=\varnothing,\qquad \dist([0_X,y],D)>0.
\end{equation}

\begin{definition}\label{def:profile}
For a radiant set \(R\) and a co-radiant set \(D\), define
\[
 \sigma_R(u):=\sup\{t\geq0:tu\in R\},\qquad
 \rho_D(u):=\inf\{t>0:tu\in D\},\qquad u\in S_X,
\]
allowing \(+\infty\) and taking \(\inf\varnothing=+\infty\).  We call $\sigma_R$ and $\rho_D$ the radiant and
co-radiant radial profiles of $R$ and $D$, respectively.
For closed \(R\), the finite radii on the ray \(\Rp u\) belonging to \(R\)
are exactly \(0\leq t\leq\sigma_R(u)\).
For \(D\), put
\(\Sigma_D=\{u\in S_X:\rho_D(u)<+\infty\}\),
\(d_D=\inf_{\Sigma_D}\rho_D\), and
\(\mathcal I_D=\sup_{\Sigma_D}\rho_D\).
For closed \(D\), its \emph{radial base}
is \(\Gamma_D=\{\rho_D(u)u:u\in\Sigma_D\}\). The set \(D\) is
\emph{directionally compact} if \(\Sigma_D\) is compact.
\end{definition}

The co-radiant radial profile is the reciprocal on \(S_X\) of the Minkowski
cogauge of \cite[Definition~2.5]{Zaffaroni2007}; see also
\cite[Section~3]{Zaffaroni2022} and Definition~\ref{def:radial-gauges} below.

\begin{proposition}\label{prop:profile}
For closed co-radiant \(D\), \(d_D=\dist(0_X,D)>0\), the function \(\rho_D\)
is lower semicontinuous on \(S_X\), every finite threshold is attained, and
\begin{equation}\label{eq:radial-representation}
 D=\sdw\Gamma_D=\{tu:u\in\Sigma_D,\ t\geq\rho_D(u)\}.
\end{equation}
If \(\Sigma_D\) is compact, every \(D\cap TB_X\), \(T>0\), is compact.
Moreover, \(\Gamma_D\) is compact if and only if \(\Sigma_D\) is compact
and \(\rho_D|_{\Sigma_D}\) is continuous.
\end{proposition}

\begin{proof}
Since \(D\) is closed and excludes \(0_X\), \(\dist(0_X,D)>0\).
On each feasible ray, dilation invariance and closedness give exactly
\([\rho_D(u),+\infty)\). Taking infima over the rays proves the formula
for \(d_D\) and \eqref{eq:radial-representation}.
If \(u_n\to u\) and \(\liminf_n\rho_D(u_n)=a<+\infty\), pass to finite
thresholds converging to \(a\). Then \(\rho_D(u_n)u_n\to au\in D\),
so \(\rho_D(u)\leq a\). This proves lower semicontinuity.

For compact \(\Sigma_D\), \(D\cap TB_X\) is a closed subset of the
compact image of \([0,T]\times\Sigma_D\) under \((t,u)\mapsto tu\).
If \(\Gamma_D\) is compact, normalization is a continuous bijection
onto \(\Sigma_D\), hence a homeomorphism. Its inverse
\(u\mapsto\rho_D(u)u\) is continuous, as is its norm \(\rho_D(u)\).
Conversely, continuity of this inverse on compact \(\Sigma_D\)
makes \(\Gamma_D\) compact.
\end{proof}

\begin{remark}\label{rem:radial-compactness}
Compactness of $\Sigma_D$ implies neither continuity
nor boundedness above of the radial profile $\rho_D$. 
\end{remark}

\section{Directional fineness and adapted generators}\label{sec:generators}
In this section, we introduce directionally fine families of cones and their adapted generators, whose directional localization and positive homogeneity are naturally suited to the radial geometry of radiant and co-radiant sets. We establish the existence of Lipschitz generators, superlinear in the convex case, and construct an explicit axial family in arbitrary normed spaces. These constructions provide the framework for our subsequent separation and approximation results.

We recall the notion of directional fineness,
introduced in \cite[Definition~2.6]{GarciaMelguizo2026}
in the context of cone separation.
\begin{definition}\label{def:fine}
A family \(\cF\) of closed cones is \emph{directionally fine} if,
for every \(u\in S_X\) and \(0<\eta<1\), some \(C\in\cF\) satisfies
\(u\in\intt C\) and \(C\subset\cone B(u,\eta)\). It is \emph{convex
directionally fine} if \(C\) can always be chosen convex.
\end{definition}

\begin{remark}
\label{rem:closed-cones}
If a family of cones is directionally fine, the proper members obtained
by taking closures or closed convex hulls still form
directionally fine families. Members that become improper are discarded.

Indeed, for \(u\in S_X\) and \(0<\eta<1\), put
\(K=\cone B(u,\eta)\). This cone is convex because \(B(u,\eta)\)
is convex, and closed by \cite[Lemma~2.1.43(iii)]{GRTZ2023},
since \(B(u,\eta)\) is closed and bounded, and
\(\dist(0_X,B(u,\eta))=1-\eta>0\).
By Hahn--Banach, choose \(f\in X^*\) with \(\|f\|=f(u)=1\).
Then \(f\geq1-\eta\) on the ball, so \(-u\notin K\) and \(K\) is proper.
Consequently,
\[
 C\subset K
 \quad\Longrightarrow\quad
 \cl C\subset K,\qquad
 \cl(\co C)\subset K.
\]
Both sets contain \(u\) in their interior and
remain proper, since \(K\) is proper. Such witnesses therefore survive
the removal of improper members.

\end{remark}

\begin{definition}\label{def:adapted}
Let \(C\) be a proper closed cone and \(u\in S_X\cap\intt C\).
A continuous positively homogeneous function \(\varphi:X\to\R\) is
\emph{adapted} to \((C,u)\) if
\begin{equation}\label{eq:adapted}
 C=\{\varphi\geq0\},\qquad \varphi(u)=1,\qquad
 \varphi(x)\leq\|x\|\quad\forall x \in X.
\end{equation}
Here positive homogeneity means \(\varphi(tx)=t\varphi(x)\) for \(t\geq0\).
Such a function is \emph{superlinear} if it is superadditive.
\end{definition}

Superlinear functions are concave and have convex superlevels. Continuous
positively homogeneous functions need not be Lipschitz, but adapted
functions can be chosen Lipschitz.

\begin{proposition}\label{prop:canonical}
Every pair \((C,u)\) as in Definition~\ref{def:adapted} admits the Lipschitz adapted function
\begin{equation}\label{eq:oriented}
 \varphi_{C,u}(x):=\min\{b_C(x)/b_C(u),\|x\|\},\qquad
 b_C(x)=\dist(x,X\setminus C)-\dist(x,C).
\end{equation}
If \(C\) is convex, put \(C^+=\{g\in X^*:g(c)\geq0,\,  \forall c\in C\}\) and define
\begin{equation}\label{eq:canonical-sl}
 \varphi^{\rm sl}_{C,u}(x):=\inf\{g(x):g\in C^+,\ g(u)=1\}.
\end{equation}
This function is finite, Lipschitz, superlinear, and  satisfies $C=\{\varphi^{\rm sl}_{C,u}\geq 0\}$.  If \(C\subset\cone B(u,\eta)\) for some \(0<\eta<1\),  then $\varphi^{\rm sl}_{C,u}$   is adapted to $(C,u)$.
Furthermore,
\begin{equation}\label{eq:axial-translation}
 \varphi^{\rm sl}_{C,u}(x+tu)=\varphi^{\rm sl}_{C,u}(x)+t,
 \qquad \forall x\in X,\ \forall t\in\R.
\end{equation}
\end{proposition}

\begin{proof}
Distances to nonempty sets are \(1\)-Lipschitz. Positive dilations
preserve \(C\) and its complement, so \(b_C\) is positively homogeneous.
Properness excludes a neighbourhood of zero in \(C\); thus \(b_C(0_X)=0\).
Moreover \(b_C(u)>0\), and closedness gives \(C=\{b_C\geq0\}\).
The minimum in \eqref{eq:oriented} is Lipschitz, has value \(1\) at \(u\),
is bounded above by the norm, and preserves this zero superlevel.

For convex \(C\), choose \(a>0\) with \(u+aB_X\subset C\).
For \(g\in C^+\), testing \(u\pm az\), \(z\in B_X\), gives
\(a\|g\|\leq g(u)\). Separation supplies a nonzero member of \(C^+\);
its value at \(u\) is positive. The normalized set in
\eqref{eq:canonical-sl} is therefore nonempty and bounded by \(1/a\).
Its infimum is finite, \(1/a\)-Lipschitz, homogeneous, and superadditive.
For \(x\notin C\), separation supplies \(g\in C^+\) with \(g(x)<0\);
normalizing \(g\) proves that the zero superlevel is \(C\).
Normalization and \eqref{eq:axial-translation} follow termwise.

If \(C\subset\cone B(u,\eta)\), choose \(f\in X^*\) with
\(\|f\|=f(u)=1\), by Hahn--Banach.
Then \(f(v)\geq1-\eta>0\) on \(B(u,\eta)\).
Hence it belongs to the normalized dual set, giving
\(\varphi^{\rm sl}_{C,u}\leq f\leq\|\cdot\|\).
\end{proof}

\begin{definition}
\label{def:localized pair}
\label{def:selected-generators}
Let \(\cF\) be a directionally fine family. A pair \((C,u)\) is
\emph{localized} if \(C\in\cF\), \(u\in S_X\cap\intt C\), and
\[
C\subset\cone B(u,\eta)
\quad\text{for some }0<\eta<1.
\]
Denote the set of localized pairs by \(\mathcal L_{\cF}\). For each
\((C,u)\in\mathcal L_{\cF}\), fix once and for all a Lipschitz positively
homogeneous function \(\varphi_{C,u}:X\to\R\) satisfying
\[
C=\{\varphi_{C,u}\geq0\},\qquad
\varphi_{C,u}(u)=1,\qquad
\varphi_{C,u}(x)\leq\|x\|\quad\forall x \in X.
\]
Each function so selected is called a \emph{selected Lipschitz adapted
generator associated with \(\cF\)}. Their family is
\[
\cG_{\cF}
:=\{\varphi_{C,u}:(C,u)\in\mathcal L_{\cF}\}.
\]
For every convex localized pair, select the canonical superlinear generator
\(\varphi^{\rm sl}_{C,u}\) from \eqref{eq:canonical-sl}. 
\end{definition}

The family of cones generated by balls gives an immediate example of directional fineness. We first describe its Euclidean form in \(\R^2\).

\begin{example}
\label{ex:euclidean-ball-generated-cones}
In \(X=\R^2\) with the Euclidean norm, write \(B_2(u,\delta)\)
for its closed balls. For \(u\in S_X\) and
\(0<\delta<1\), let
\[
 C_{u,\delta}:=\cone B_2(u,\delta),\qquad
 \cF_{\mathrm{Euc}}:=\{C_{u,\delta}:u\in S_X,\ 0<\delta<1\}.
\]
Writing \(Q_u x=x-\langle x,u\rangle u\), the associated selected
adapted generator is
\[
 \varphi_{u,\delta}(x)
 :=\langle x,u\rangle
 -\frac{\sqrt{1-\delta^2}}{\delta}\|Q_u x\|_2,
 \qquad C_{u,\delta}=\{\varphi_{u,\delta}\geq0\}.
\]
It is Lipschitz and superlinear, with \(\varphi_{u,\delta}(u)=1\) and
\(\varphi_{u,\delta}\leq\|\cdot\|_2\). The family is directionally fine:
for any \(u\in S_X\) and \(0<\eta<1\), choose \(0<\delta<\eta\), so
\(u\in\intt C_{u,\delta}\subset\cone B_2(u,\eta)\).
\end{example}

The ball-generated cones remain convex and directionally fine in arbitrary normed spaces. However, the explicit Euclidean formula for their generators relies on the inner-product structure and does not directly extend to general norms. Indeed, let \(u\in S_X\), \(0<\delta<1\), and
\(C=\cone B(u,\delta)\), a closed convex cone. Its dual cone satisfies
\(g\in C^+\) iff \(g(u)-\delta\|g\|\geq0\). Every nonzero
\(g\in C^+\) has \(g(u)>0\), so it can be normalized by \(g(u)=1\).
Separation then gives the canonical generator and its zero-superlevel
representation:
\[
 \varphi^{\rm sl}_{C,u}(z)
 :=\inf\{g(z):g\in X^*,\ g(u)=1,\ \delta\|g\|\leq1\},
 \qquad
 C=\{z\in X:\varphi^{\rm sl}_{C,u}(z)\geq0\}.
\]
The preceding ball-generated family is directionally fine in arbitrary normed spaces, but its canonical superlinear generators are described in general by a dual-infimum formula. The axial cones below are the
uniform axial deviation cones of
\cite[Corollary~5.6]{GarciaMelguizo2026}. We use their generators at
nonzero radial levels and subsequently apply norm correction to obtain
bilateral separation for radiant and co-radiant sets.

Let $u\in S_X$ and $f\in S_{X^*}$ satisfy $f(u)=1$.
We call $(u,f)$ a norming pair. Set
\[
Q_{u,f}:=I_X-u\otimes f,
\]
where $I_X$ is the identity on $X$ and
$(u\otimes f)x=f(x)u$. Thus,
\[
Q_{u,f}x=x-f(x)u,\qquad
x=f(x)u+Q_{u,f}x,\qquad \forall x\in X.
\]
The operator $Q_{u,f}$ is the projection onto
$\ker f$ along $\R u$.
\begin{definition}\label{def:uad}
Let $(u,f)$ be a norming pair and $\delta>0$.
The axial cone and its associated axial function
are defined by
\[
C_{u,f,\delta}
:=\{x\in X:\|Q_{u,f}x\|\leq\delta f(x)\},
\]
and
\[
\varphi_{u,f,\delta}(x)
:=f(x)-\frac{1}{\delta}\|Q_{u,f}x\|,
\qquad \forall x\in X.
\]
The families of axial cones and axial functions are,
respectively,
\[
\mathcal F_{\rm ax}
:=\{C_{u,f,\delta}:
(u,f)\text{ is a norming pair},\ \delta>0\},
\]
\[
\mathcal G_{\rm ax}
:=\{\varphi_{u,f,\delta}:
(u,f)\text{ is a norming pair},\ \delta>0\}.
\]

\end{definition}

\begin{proposition}\label{prop:uad}
 The family  \(\cF_{\AX}\) is  a convex directionally fine family in every real normed
space.  The functions  in
$\cG_{\AX}$ coincide with the canonical superlinear generators
\(\varphi^{\rm sl}_{C_{u,f,\delta},u}\)  defined by \eqref{eq:canonical-sl},
are adapted, and satisfy
\begin{equation}\label{eq:uad-shift}
 \begin{gathered}
 \varphi_{u,f,\delta}\leq f\leq\|\cdot\|,\qquad
 \Lip(\varphi_{u,f,\delta})\leq1+\|Q_{u,f}\|/\delta,\quad \|Q_{u,f}\|\leq2,\\
 \varphi_{u,f,\delta}(x+tu)=\varphi_{u,f,\delta}(x)+t,\qquad \forall t\in\R.
 \end{gathered}
\end{equation}
For \(0<\delta<\eta<1\),
\begin{equation}\label{eq:uad-narrow}
 C_{u,f,\delta}\subset\cone B(u,\eta).
\end{equation}
\end{proposition}

\begin{proof}
The inequality \(\|Q_{u,f}x\|\leq\|x\|+|f(x)|\leq2\|x\|\)
gives the operator bound. Subadditivity of \(\|Q_{u,f}(\cdot)\|\)
gives superlinearity; the triangle inequality gives the Lipschitz bound.
Since \(Q_{u,f}u=0_X\), the translation identity and adaptation follow.
The zero superlevel is closed and convex, contains \(u\) in its interior,
and excludes \(-u\).

Write \(\varphi=\varphi_{u,f,\delta}\). The translation identity gives
\(x-\varphi(x)u\in C_{u,f,\delta}\), so every normalized dual functional
satisfies \(g(x)\geq\varphi(x)\).
By Hahn--Banach, choose \(h\in X^*\), \(\|h\|\leq1\), with
\(h(Q_{u,f}x)=\|Q_{u,f}x\|\). Then
\(g=f-\delta^{-1}h\circ Q_{u,f}\) satisfies \(g(u)=1\) and
\(g(c)\geq f(c)-\delta^{-1}\|Q_{u,f}c\|\geq0\) on the cone.
Also \(g(x)=\varphi(x)\), proving the canonical formula.

For nonzero \(x\) in the cone, \(f(x)>0\) and
\(x/f(x)\in B(u,\delta)\), proving \eqref{eq:uad-narrow}.
For each \(u\) and \(\eta\), Hahn--Banach supplies a norming \(f\);
choosing \(0<\delta<\eta\) proves directional fineness.
\end{proof}

For a fixed norming pair $(u,f)$ , the nonzero vectors in
\(C_{u,f,\delta}\) are precisely those with positive axial component
\(f(x)\) and relative transverse deviation
\(\|Q_{u,f}x\|/f(x)\leq\delta\).
The parameter \(\delta\) therefore bounds the transverse deviation
uniformly along each ray. The next lemma identifies the generating section.

\begin{lemma}
\label{lem:axial-cone-representation}
Let \((u,f)\) be a norming pair and \(\delta>0\). Then
\[
 C_{u,f,\delta}
 =\cone\bigl((u+\ker f)\cap B(u,\delta)\bigr).
\]
\end{lemma}
\begin{proof}
For the forward inclusion, let \(x\in C_{u,f,\delta}\) and put \(t=f(x)\).
The inequality \(\|x-tu\|\leq\delta t\) implies \(t\geq0\).
If \(t=0\), then \(x=0_X\). Otherwise \(z=(x-tu)/t\) satisfies
\(f(z)=0\) and \(\|z\|\leq\delta\).
Thus \(u+z\in(u+\ker f)\cap B(u,\delta)\) and \(x=t(u+z)\).
Conversely, any point of the generated cone has this form with
\(t\geq0\), \(z\in\ker f\), and \(\|z\|\leq\delta\).
Since \(f(u)=1\), one has \(f(x)=t\) and
\(\|x-f(x)u\|=t\|z\|\leq\delta f(x)\), proving membership in the axial cone.
\end{proof}

In a Hilbert space, the ball-generated and axial families coincide up to
a reparametrization. Indeed, for \(u\in S_X\), \(f(x)=\langle x,u\rangle\),
and \(0<\beta<1\), one has
\(\cone B(u,\beta)=C_{u,f,\delta}\) with
\(\delta=\beta/\sqrt{1-\beta^2}\); the corresponding canonical generators
also agree.

Unless explicitly stated otherwise, every concrete cone family used in
the examples below is the axial family \(\cF_{\AX}\).

\section{Positive-level separation}\label{sec:positive}

In this section, we show that positive superlevels of adapted generators can be localized within the dilation shadows of small balls. We use this construction to separate points from closed radiant sets and, by compactness, to obtain finite separation of compact sets. Finally, we give an explicit version of the point-separation result using the axial family.

\begin{lemma}\label{lem:directions}
If \(u\in S_X\), \(0<\eta<1\), and \(v\in S_X\cap\cone B(u,\eta)\), then
\(\|v-u\|\leq2\eta\). For \(a\ne0_X\) and \(\|h\|<\|a\|\),
\begin{equation}\label{eq:normalization}
 \left\|\frac{a+h}{\|a+h\|}-\frac a{\|a\|}\right\|
 \leq\frac{2\|h\|}{\|a\|}.
\end{equation}
\end{lemma}

\begin{proof}
Insert \((a+h)/\|a\|\) between the two normalized vectors.
The two resulting distances are
\(\bigl|\|a+h\|-\|a\|\bigr|/\|a\|\) and \(\|h\|/\|a\|\).
The reverse triangle inequality bounds the first by the second,
proving \eqref{eq:normalization}. For the first assertion,
write \(v=(u+h)/\|u+h\|\), with \(\|h\|\leq\eta<1\), and apply
that estimate with \(a=u\).
\end{proof}

The preceding lemma controls the directions of vectors in cones localized around a given unit vector. We now use this estimate to show that positive superlevels of sufficiently localized adapted generators lie within a prescribed dilation shadow around a point, while containing a smaller dilation shadow.
\begin{theorem}\label{thm:local}
Let $\cF$ be a directionally fine family
with selected adapted generators $\cG_{\cF}$.
Let $y=su$, $s>0$, $u\in S_X$, and
$0<\varepsilon<s$. Choose
\begin{equation}\label{eq:local-parameters}
s-\varepsilon<\lambda<s,\qquad
\gamma=\lambda-(s-\varepsilon),\qquad
0<\eta<\min\left\{1,\frac{\gamma}{2\lambda}\right\}.
\end{equation}
For every $C\in\cF$ such that
$u\in\intt C$ and $C\subset\cone B(u,\eta)$,
let $\varphi=\varphi_{C,u}\in\cG_{\cF}$
be its selected adapted generator. Then
\begin{equation}\label{eq:local-sandwich}
y\in\{\varphi>\lambda\}
\subset\{\varphi\geq\lambda\}
\subset\sdw B(y,\varepsilon).
\end{equation}
The strict and weak superlevels are co-radiant,
respectively open and closed.
If $L=\Lip(\varphi)$, then
\[
\sdw B(y,r)\subset\{\varphi>\lambda\}
\qquad
\text{for }0<r<\frac{s-\lambda}{L}.
\]

If $C$ is convex, then $\varphi$
is superlinear and both superlevels are
convex. In particular, if $\cF$ is convex
directionally fine, such a convex cone $C$
can be chosen at every prescribed direction
and aperture.

\end{theorem}

\begin{proof}
Since \(\varphi(y)=s>\lambda\), the first inclusion holds.
If \(\varphi(x)\geq\lambda\), adaptation gives
\(x\in C\setminus\{0_X\}\) and \(\|x\|\geq\lambda\).
For \(v=x/\|x\|\), Lemma~\ref{lem:directions} yields
\(\|\lambda v-y\|\leq\lambda\|v-u\|+s-\lambda
\leq2\lambda\eta+s-\lambda<\varepsilon\).
Hence \(x=(\|x\|/\lambda)(\lambda v)\in\sdw B(y,\varepsilon)\).
Continuity gives openness and closedness of the respective superlevels;
positive homogeneity gives co-radiancy. If \(\varphi\) is superlinear,
its concavity gives convexity. Finally,
\(\varphi(z)\geq s-Lr>\lambda\) on \(B(y,r)\).
Positive dilation preserves this strict inequality, proving the inner inclusion.
\end{proof}

\begin{remark}\label{rem:local-order}
The order of choices in Theorem 4.1 is important. The outer radius \(\varepsilon\) is prescribed first, whereas the inner radius \(r\) can only be determined after choosing \(\lambda\), \(\eta\), and the selected generator \(\varphi\), since \(r<(s-\lambda)/\operatorname{Lip}(\varphi)\). In particular, directional fineness alone does not provide a positive inner radius uniform over all admissible generators.
\end{remark}
The dilation shadow of a ball disjoint from a radiant set remains disjoint from that set. Combining this observation with Theorem 4.1 yields the following point-separation result.
\begin{theorem}\label{thm:radiant-point}
Let \(\cF\) be directionally fine, let \(R\subset X\) be closed and
radiant, and let \(y\notin R\). There exist
\(\varphi\in\cG_{\cF}\) and \(\lambda>0\) such that, for
\(p:=\varphi/\lambda\),
\begin{equation}\label{eq:radiant-point}
 y\in\{p>1\}\subset\{p\geq1\}\subset X\setminus R.
\end{equation}
Both superlevels are co-radiant neighbourhoods of \(y\); the strict one
is open and the weak one is closed. The function \(p\) is Lipschitz and it can be chosen so that \( p(y)>1>\sup_{r\in R}p(r).\)
If \(\cF\) is convex directionally fine, then \(\varphi\) may be chosen
superlinear, and both superlevels are convex.
\end{theorem}

\begin{proof}
Since \(0_X\in R\) and \(R\) is closed, choose
\(0<\varepsilon<\min\{\|y\|,\dist(y,R)\}\).
The ball \(B(y,\varepsilon)\) avoids \(R\), and so does its dilation
shadow: if \(tz\in R\), with \(t\geq1\) and \(z\in B(y,\varepsilon)\),
radiancy would give \(z\in R\).
Theorem~\ref{thm:local} supplies \(\varphi\) and \(\lambda>0\) with
\(y\in\{\varphi>\lambda\}\subset\{\varphi\geq\lambda\}
\subset\sdw B(y,\varepsilon)\).
For \(p=\varphi/\lambda\), these are precisely the inclusions
in \eqref{eq:radiant-point}. The regularity and superlevel properties
follow from the same theorem.   
Finally, choose $c$ with $1/p(y)<c<1$ and
replace $p$ by $cp$. Then $cp(y)>1>\sup_{r\in R}cp(r)$, as required.

\end{proof}

The preceding point-separation result extends to compact sets: finitely many adapted generators can be combined into a single Lipschitz separator.
\begin{theorem}\label{thm:compact-positive}
Let $\cF$ be a directionally fine family
with selected adapted generators $\cG_{\cF}$.
Let $R\subset X$ be closed and radiant, and let
$K\subset X\setminus R$ be nonempty and compact.
Then there exist $m\in\mathbb N$, \(\varphi_i\in\cG_{\cF}\) and
\(\lambda_i>0\), \(1\leq i\leq m\), such that
\[
 K\subset\bigcup_i\{\varphi_i>\lambda_i\}
 \subset\bigcup_i\{\varphi_i\geq\lambda_i\}\subset X\setminus R.
\]
The positively homogeneous Lipschitz function
\begin{equation}\label{eq:max-positive}
 \Psi:=\max_i\varphi_i/\lambda_i
\end{equation}
satisfies \(\min_K\Psi>1\), \(\Psi(r)<1\) for \(r\in R\), and
\(\Lip(\Psi)\leq\max_i\Lip(\varphi_i)/\lambda_i\).
\end{theorem}

\begin{proof}
For each \(y\in K\), Theorem~\ref{thm:radiant-point} supplies a
strict superlevel containing \(y\) whose weak superlevel avoids \(R\).
Compactness gives finitely many such strict superlevels covering \(K\).
Their maximum \(\Psi\) is pointwise above \(1\) on \(K\) and below
\(1\) on \(R\). It is continuous, so its attained minimum on \(K\)
is strictly greater than \(1\).
The estimate
\(|\max_i p_i(x)-\max_i p_i(z)|\leq\max_i|p_i(x)-p_i(z)|\),
with \(p_i=\varphi_i/\lambda_i\), proves the Lipschitz bound.
Positive homogeneity passes to this finite maximum.
\end{proof}

\begin{remark}\label{rem:max}
Even when the selected generators are superlinear, their finite maximum \(\Psi\) in \eqref{eq:max-positive} need not be superlinear. Consequently, the union of their convex superlevels need not be convex.
\end{remark}

\begin{remark}\label{rem:obstruction}
 The adapted generators used above cannot, in general, provide negative-level separation of co-radiant sets. Indeed, consider \(D=\{x\in X:\|x\|\geq1\}\), and let \(\varphi\) be an adapted generator associated with a direction \(u\in S_X\). Since \(\varphi(0_X)=0\) and \(\varphi(u)=1\), while \(u\in D\), it is impossible to have
\[
\varphi(0_X)>-1\geq\sup_{d\in D}\varphi(d).
\]
This obstruction motivates the norm-corrected generators introduced in Section 5.
\end{remark}

We now specialize the point-separation theorem to the axial family, obtaining explicit separators with quantitative bounds on their parameters.

\begin{corollary}\label{cor:uad-positive}
Let $R\subset X$ be closed and radiant,
$y=su\notin R$, with $s>0$ and $u\in S_X$.
Choose
\[
0<\theta<\min\{\dist(y,R),s\},\qquad
s-\theta<\lambda<s,\qquad
0<\delta<\min\left\{
1,\frac{\theta-s+\lambda}{2\lambda}
\right\}.
\]
For any norming pair $(u,f)$, consider
the axial cone $C_{u,f,\delta}\in\cF_{\AX}$
and its associated generator
$\varphi_{u,f,\delta}\in\cG_{\AX}$.
Set $p:=\varphi_{u,f,\delta}/\lambda$.
Then
\[
y\in\{p>1\}
\subset\{p\geq1\}
\subset X\setminus R.
\]
The strict and weak superlevels are
co-radiant neighbourhoods of $y$.
\end{corollary}

\begin{proof}
Write \(\varphi=\varphi_{u,f,\delta}\).
Its normalization gives \(\varphi(y)=s>\lambda\).
If \(\varphi(x)\geq\lambda\), then \(\|x\|\geq\lambda\) and
\(v=x/\|x\|\in C_{u,f,\delta}\cap S_X\).
Proposition~\ref{prop:uad} and Lemma~\ref{lem:directions} give
\(\|v-u\|\leq2\delta\).
Consequently \(\|\lambda v-y\|\leq2\lambda\delta+s-\lambda<\theta\),
and \(x=(\|x\|/\lambda)(\lambda v)\in\sdw B(y,\theta)\).
This shadow avoids \(R\) by radiancy, since \(B(y,\theta)\) does.
The inclusions follow. Continuity and positive homogeneity give
the asserted co-radiant neighbourhoods.
\end{proof}

\section{Bilateral separation}\label{sec:bilateral}

 The positive-level separation results of Section \ref{sec:positive} do not extend directly to negative levels using unmodified adapted generators. We overcome this limitation by introducing norm-corrected generators, which yield separation results for both radiant and co-radiant sets. We then use radial calibration to characterize admissible separating levels and obtain representations of the gauge and cogauge. Finally, we specialize these constructions to the axial family, obtaining explicit separators.

\begin{definition}\label{def:norm-corrected}
Let $\cF$ be a directionally fine family of closed cones, and let
\(\cG_{\cF}\) be its family of selected Lipschitz adapted generators
from Definition~\ref{def:selected-generators}.
For \(\varphi\in\cG_{\cF}\), \(\alpha\geq0\), and \(\ell>0\), set
\[
P_{\alpha,\ell,\varphi}:=\dfrac{\alpha\varphi-\|\cdot\|}{\ell}.
\]

The associated \emph{norm-corrected family} is
\begin{equation}\label{eq:norm-corrected}
 \cP_{\cF}^{\nc}
 :=\left\{
 P_{\alpha,\ell,\varphi}:
 \alpha\geq0,\ \ell>0,\ \varphi\in\cG_{\cF}
 \right\}.
\end{equation}
\end{definition}

The parameter $\alpha$ controls the balance between the generator and
the norm, determining the sign regime. The parameter $\ell$ rescales
the nonzero levels without changing the zero superlevel.

The following proposition establishes the basic properties
of the norm-corrected functions and describes the two regimes
$\alpha>1$ and $0\leq\alpha<1$.

\begin{proposition}\label{prop:regimes}
Let $\cF$ be a directionally fine family,
and let $\cP_{\cF}^{\rm nc}$ be the norm-corrected
family defined in \eqref{eq:norm-corrected}.
Every $P=P_{\alpha,\ell,\varphi}
\in\cP_{\cF}^{\rm nc}$
is positively homogeneous and Lipschitz,
with $\Lip(P)\leq
(\alpha\Lip(\varphi)+1)/\ell$, and is superlinear when $\varphi$ is.
\begin{enumerate}
\item
For fixed $\alpha>1$, the
family of cones $
\{C_{\alpha,\varphi}:
\varphi\in\cG_{\cF}\}
$, $C_{\alpha,\varphi}:=\{\alpha\varphi-\|\cdot\|\geq0\}$,
is
a directionally fine family.
The function
\begin{equation}\label{eq:H-normalization}
H_{\alpha,\varphi}
=\frac{\alpha\varphi-\|\cdot\|}{\alpha-1}
\end{equation}
is adapted
to $C_{\alpha,\varphi}$ at $u$
whenever $\varphi=\varphi_{C,u}$
is the selected generator of a localized
pair $(C,u)$.
For superlinear generators,
the corrected cones are convex.

\item For $0\leq\alpha<1$,
the following inequality holds:
\begin{equation}\label{eq:coercive-negative}
P(x)\leq-\frac{1-\alpha}{\ell}\|x\|,
\qquad \forall x\in X.
\end{equation}
Thus $\{P\geq0\}=\{0_X\}$,
and $\{P>-1\}$, $\{P\geq-1\}$
are bounded radiant neighbourhoods of zero,
respectively open and closed.
They are convex when $\varphi$ is superlinear.
\end{enumerate}
\end{proposition}

\begin{proof}
Positive homogeneity follows termwise. The reverse triangle
inequality gives the Lipschitz estimate, and the sum of the superlinear
functions \(\alpha\varphi\) and \(-\|\cdot\|\) is superlinear.

For \textup{(i)}, the numerator equals \(\alpha-1>0\) at \(u\).
Continuity therefore places \(u\) in the interior of its zero superlevel.
This cone is contained in \(\{\varphi\geq0\}\), so it is proper
and inherits the selected cone's localization.
Moreover, \(H_{\alpha,\varphi}(u)=1\) and
\(\alpha\varphi(x)-\|x\|\leq(\alpha-1)\|x\|\).
Thus \(H_{\alpha,\varphi}\) is adapted.
Directional fineness follows by choosing the original cones at
arbitrarily small apertures. Concavity gives convexity in the
superlinear case.

For \textup{(ii)}, the same inequality proves
\eqref{eq:coercive-negative}. Put \(c=(1-\alpha)/\ell>0\).
If \(P(x)\geq-1\), then \(-1\leq P(x)\leq-c\|x\|\), so
\(\|x\|\leq1/c\); strict superlevels satisfy the strict bound.
The estimate also gives \(\{P\geq0\}=\{0_X\}\).
Since \(P(0_X)=0>-1\), continuity gives a neighbourhood of zero.
For \(0\leq t\leq1\), \(P(x)\geq-1\) implies
\(P(tx)=tP(x)\geq-1\), and likewise for the strict superlevel.
This proves radiancy. Continuity and concavity give the remaining properties.
\end{proof}

\begin{remark}
The parameter $\alpha=1$ separates two different
geometric regimes. Indeed,
since $\varphi(u)=1$ and $\varphi\leq\|\cdot\|$, the corrected function
$P=(\alpha\varphi-\|\cdot\|)/\ell$ satisfies
\[
 P(u)=\frac{\alpha-1}{\ell},
 \qquad
 P(x)\leq\frac{\alpha-1}{\ell}\|x\|.
\]
If $0\leq\alpha<1$, the negative superlevel $\{P\geq-1\}$ is
bounded. If $\alpha>1$,
$P(tu)\to+\infty$ as $t\to+\infty$.
At $\alpha=1$, $P\leq0$, while $P(tu)=0$ for every $t\geq0$:
all superlevels at strictly positive levels are empty
and $\{P\geq-1\}$ is unbounded.
\end{remark}

We now use these two regimes to separate points from closed radiant and co-radiant sets at positive and negative levels, respectively.
\begin{theorem}\label{thm:bilateral}
Let $\cF$ be a directionally fine family of
closed cones in $X$, let $\cG_{\cF}$ be
its family of selected adapted generators,
and let $\cP_{\cF}^{\rm nc}$ be the associated
norm-corrected family.
The following assertions hold.
\begin{enumerate}
\item If $R$ is closed and radiant and $y\notin R$,
then, for every fixed $\alpha>1$, some
$\varphi\in\cG_{\cF}$ and $\ell>0$ exist such that,
with $P:=P_{\alpha,\ell,\varphi}
\in\cP_{\cF}^{\rm nc}$,
\begin{equation}\label{eq:bilateral-positive}
 y\in\{P>1\}\subset\{P\geq1\}\subset X\setminus R.
\end{equation}
The strict and weak superlevels are co-radiant
neighbourhoods of $y$.

\item If $D$ is closed and co-radiant and $y\notin D$,
then some $0\leq\alpha<1$, $\ell>0$, and
$\varphi\in\cG_{\cF}$ exist such that,
with $P:=P_{\alpha,\ell,\varphi}
\in\cP_{\cF}^{\rm nc}$,
\begin{equation}\label{eq:bilateral-negative}
 y\in\{P>-1\}\subset\{P\geq-1\}\subset X\setminus D.
\end{equation}
The strict and weak superlevels are bounded
radiant neighbourhoods of $0_X$ and $y$.
For $y=0_X$, one may take $\alpha=0$ and
$0<\ell<d_D$.
\end{enumerate}
In both cases the strict superlevel is open and
the weak superlevel is closed.
If $\mathcal F$ is convex directionally fine, the localized
witnesses in both constructions may be chosen convex. With their
canonical superlinear generators, the separating functions are
superlinear and their superlevel sets are convex.
\end{theorem}

\begin{proof}
For \textup{(i)}, apply the proof of Theorem~\ref{thm:radiant-point} to the corrected
cones and adapted functions \(H_{\alpha,\varphi}\) of
Proposition~\ref{prop:regimes}. If the resulting level is \(\lambda>0\),
set \(\ell=(\alpha-1)\lambda\). Then
\[
\frac{H_{\alpha,\varphi}}{\lambda}
=P_{\alpha,\ell,\varphi},
\]
and Theorem~\ref{thm:radiant-point} gives
\(P_{\alpha,\ell,\varphi}(y)>1>P_{\alpha,\ell,\varphi}(r)\) for \(r\in R\).
Hence
\[
y\in\{P_{\alpha,\ell,\varphi}>1\}
\subset\{P_{\alpha,\ell,\varphi}\geq1\}
\subset X\setminus R.
\]

For \textup{(ii)}, let \(y=su\neq0_X\). Choose \(a>s\), sufficiently close
to \(s\), so that \(au\notin D\), put \(b=\dist(au,D)>0\), and select
\begin{equation}
\label{eq:abstract-cap-parameters}
0<\eta<\min\left\{1,\frac{b}{2a}\right\},\qquad
0<\ell<\min\{d_D,a\},\qquad
\alpha=1-\frac{\ell}{a}.
\end{equation}
Choose \(C\subset\cone B(u,\eta)\) and its selected generator \(\varphi\).
Then
\begin{equation}
\label{eq:truncated-cone-empty}
D\cap C\cap aB_X=\varnothing.
\end{equation}
Indeed, if \(d\) belonged to this intersection and
\(v=d/\|d\|\), then co-radiancy would give \(av\in D\), whereas
\[
\|av-au\|\leq2a\eta<b,
\]
contradicting the definition of \(b\).

Let \(P=P_{\alpha,\ell,\varphi}\). Since
\[
P(y)=\frac{(\alpha-1)s}{\ell}=-\frac{s}{a}>-1,
\]
one has \(y\in\{P>-1\}\). If \(P(x)\geq-1\), then
\[
\|x\|\leq\ell+\alpha\varphi(x)
\leq\ell+\alpha\|x\|,
\]
and therefore \(\|x\|\leq a\). If \(x\in C\),
\eqref{eq:truncated-cone-empty} gives \(x\notin D\). If \(x\notin C\),
then \(\varphi(x)<0\), so
\[
\|x\|\leq\ell+\alpha\varphi(x)<\ell<d_D,
\]
and again \(x\notin D\). Thus
\[
y\in\{P>-1\}\subset\{P\geq-1\}\subset X\setminus D.
\]

For \(y=0_X\), take \(\alpha=0\) and \(0<\ell<d_D\); then
\(P=-\|\cdot\|/\ell\) and
\(\{P\geq-1\}=\ell B_X\subset X\setminus D\).
Proposition~\ref{prop:regimes} gives all remaining assertions.
\end{proof}
The following example illustrates both separation regimes of Theorem 5.1 through explicit constructions in the Euclidean plane.
\begin{example}
\label{ex:bilateral-axial}
In Euclidean \(\R^2\), let \(f_1(x)=x_1\) and \(f_2(x)=x_2\), and put
\(\varphi_1=\varphi_{e_1,f_1,1}=x_1-|x_2|\) and
\(\varphi_2=\varphi_{e_2,f_2,1}=x_2-|x_1|\).
\begin{enumerate}
\item Let
\(R=\{x\in\R^2:x_1\leq |x_2|/(1+|x_2|)\}\) and
\(y_R=(1/2,0)\). The set \(R\) is closed and radiant. Set
\(P_R=P_{2,1/6,\varphi_1}=6(2\varphi_1-\|\cdot\|_2)\). For \(r\in R\),
\(\varphi_1(r)\leq0\), hence \(P_R(r)\leq-6\|r\|_2\leq0\). Thus
\[
 P_R(y_R)=3>1>\max_{r\in R}P_R(r)=0,
 \qquad \{P_R\geq1\}\subset\R^2\setminus R.
\]
The positive superlevel is co-radiant and unbounded.
\item Let \(D=\{x\in\R^2:\|x\|_2\geq1\}\) and \(y_D=(0,1/2)\).
Set \(P_D=P_{0,3/4,\varphi_2}=-(4/3)\|\cdot\|_2\). Then
\[
 P_D(y_D)=-\frac23>-1>-\frac43=\max_{d\in D}P_D(d),
 \qquad \{P_D\geq-1\}=\{x:\|x\|_2\leq3/4\}\subset\R^2\setminus D.
\]
The weak superlevel is a bounded radiant neighbourhood of \(0_X\)
and \(y_D\).
\end{enumerate}
\end{example}
The preceding point-separation results also yield finite separation of sets under suitable compactness assumptions.

\begin{corollary}\label{cor:compact-negative}
Let $\cF$ be a directionally fine family of 
closed cones in $X$, with selected adapted generators
$\cG_{\cF}$ and associated norm-corrected family
$\cP_{\cF}^{\nc}$.
Let $R$ be closed and radiant and $D$ closed and co-radiant,
with $R\cap D=\varnothing$. Then:
\begin{enumerate}
\item
If $R$ is compact,
there exist $m\in\mathbb{N}$ and functions
$P_i:=P_{\alpha_i,\ell_i,\varphi_i}
\in\cP_{\cF}^{\nc}$, with
$\varphi_i\in\cG_{\cF}$, $\ell_i>0$ and
$0\leq\alpha_i<1$, $1\leq i\leq m$, such that
$\Psi_R:=\max_{1\leq i\leq m}P_i$ satisfies
\[
 \min_R\Psi_R>-1>\Psi_R(d)
 \qquad \forall d\in D.
\]
Its strict and weak superlevels at $-1$ are bounded
and radiant; the weak superlevel is disjoint from $D$.
The assertion also holds with any nonempty compact
$K\subset X\setminus D$ in place of $R$.

\item
If the radial base $\Gamma_D$ is compact, then,
for any fixed $\alpha>1$,
there exist $n\in\mathbb{N}$ and functions
$Q_j:=P_{\alpha,\ell_j,\varphi_j}
\in\cP_{\cF}^{\nc}$, with
$\varphi_j\in\cG_{\cF}$ and $\ell_j>0$,
$1\leq j\leq n$, such that
$\Psi_D:=\max_{1\leq j\leq n}Q_j$ satisfies
\[
 \Psi_D(r)<1<
 \min_{a\in\Gamma_D}\Psi_D(a)
 \leq\Psi_D(d)
 \qquad \forall r\in R,\ \forall d\in D.
\]
Its strict and weak superlevels at $1$ are
co-radiant; the weak superlevel is disjoint from $R$.
\end{enumerate}
In either case, the resulting function is
Lipschitz and positively homogeneous.
Its strict superlevel is open and its weak superlevel
is closed.
\end{corollary}

\begin{proof}
For \textup{(i)}, apply Theorem~\ref{thm:bilateral}(ii) at every
point of \(R\). The strict superlevels form an open cover;
compactness gives a finite subcover.
Their maximum \(\Psi_R\) is continuous and exceeds \(-1\) on \(R\),
so its attained minimum also exceeds \(-1\).
All selected functions are below \(-1\) on \(D\), hence so is their
finite maximum. Moreover,
\(\Psi_R\leq-(\min_i c_i)\|\cdot\|\), where
\(c_i=(1-\alpha_i)/\ell_i>0\), proving boundedness of the superlevels.
The same argument applies to any compact \(K\subset X\setminus D\).

For \textup{(ii)}, fix \(\alpha>1\) and apply
Theorem~\ref{thm:bilateral}(i) at each \(a\in\Gamma_D\).
A finite subcover gives \(\min_{\Gamma_D}\Psi_D>1\) and
\(\Psi_D(r)<1\) for \(r\in R\).
Every \(d\in D\) has the form \(d=ta\), with \(a\in\Gamma_D\),
\(t\geq1\); thus \(\Psi_D(d)=t\Psi_D(a)\geq\min_{\Gamma_D}\Psi_D\).
Finite maxima preserve Lipschitz continuity and positive homogeneity.
These properties give the remaining superlevel assertions in both parts.
\end{proof}
The following example illustrates the finite separation result of Corollary \ref{cor:compact-negative} and shows that the resulting superlevel set need not be convex, even when the individual superlevels are convex.
\begin{example}\label{ex:compact-negative}
In Euclidean \(\R^2\), let \(D=\{x:\|x\|_2\geq1\}\) and
\(R=[0,y_1]\cup[0,y_2]\), where \(y_1=(1/2,0)\) and
\(y_2=(0,1/2)\). The set \(R\) is compact, radiant, and disjoint from
\(D\). Let \(f_i(x)=x_i\) and
\(\varphi_i=\varphi_{e_i,f_i,1}\), so that
\(\varphi_1(x)=x_1-|x_2|\) and \(\varphi_2(x)=x_2-|x_1|\).
Set \(P_i=P_{1/2,3/10,\varphi_i}\) and \(\Psi=\max\{P_1,P_2\}\).
Then \(P_i(y_i)=-5/6\), \(P_i(y_j)=-5/2\) for \(i\ne j\), and
\(P_i(x)\leq-(5/3)\|x\|_2\). By homogeneity,
\(\min_R\Psi=-5/6>-1\), and
\(\Psi<-1\) on \(D\). The weak superlevel is \(G_1\cup G_2\), where
\[
 G_1=\{x:\|x\|_2+(|x_2|-x_1)/2\leq3/10\},\qquad
 G_2=\{x:\|x\|_2+(|x_1|-x_2)/2\leq3/10\}.
\]
Both sets are convex and contain zero in their interior, and
\(G_1\cup G_2\subset(3/5)B_X\). Their union is not convex:
\(m=(1/4,1/4)\) satisfies
\(P_1(m)=P_2(m)=-5\sqrt2/6<-1\).
Nor is \(\Psi\) superlinear, since
\[
 \Psi(e_1+e_2)=-10\sqrt2/3<-10/3=\Psi(e_1)+\Psi(e_2).
\]
\end{example}

Positive rescaling preserves the norm-corrected family:
\(cP_{\alpha,\ell,\varphi}=P_{\alpha,\ell/c,\varphi}\) for \(c>0\).
We first use this property to normalize separators along
prescribed rays. This radial calibration will then allow us to
characterize admissible separating levels through \(\sigma_R\)
and \(\rho_D\), and to obtain representations of the Minkowski
gauge and cogauge in terms of strict separators.

\begin{corollary}\label{cor:radial-calibration}
Let \(\cF\) be a directionally fine family of closed
cones in \(X\), with selected adapted generators \(\cG_{\cF}\)
and associated norm-corrected family \(\cP_{\cF}^{\nc}\).
Let \(u\in S_X\), \(s,\lambda>0\), and \(y=su\).
\begin{enumerate}
\item Let \(R\) be closed and radiant. If \(0<\varepsilon<1\) and
\(\sigma_R(u)<(1-\varepsilon)\lambda<\lambda<s\), then, for every
fixed \(\alpha>1\), there exist \(\varphi\in\cG_{\cF}\) and \(\ell>0\)
such that
\(p:=P_{\alpha,\ell,\varphi}\in\cP_{\cF}^{\nc}\)
satisfies
\[
 p(tu)=t/\lambda\quad(t\geq0),\qquad
 \sup_Rp\leq1-\varepsilon<1<p(y).
\]

\item Let \(D\) be closed and co-radiant. If \(\varepsilon>0\) and
\(s<\lambda<(1+\varepsilon)\lambda<\rho_D(u)\), then
there exist \(\varphi\in\cG_{\cF}\), \(0\leq\alpha<1\)
and \(\ell>0\) such that
\(p:=P_{\alpha,\ell,\varphi}\in\cP_{\cF}^{\nc}\) satisfies
\[
 p(tu)=-t/\lambda\quad(t\geq0),\qquad
 \sup_Dp\leq-1-\varepsilon<-1<p(y).
\]
\end{enumerate}
If, in addition, \(\cF\) consists of convex cones,
the separators in both cases are superlinear.
\end{corollary}

\begin{proof}
For \textup{(i)}, apply Theorem~\ref{thm:bilateral}(i) with the
prescribed \(\alpha>1\) at \((1-\varepsilon)\lambda u\), which lies
outside \(R\). The resulting \(q\) satisfies
\((1-\varepsilon)\lambda q(u)>1\) and \(\sup_Rq\leq1\).
In particular \(q(u)>0\). Set \(p=q/(\lambda q(u))\).
Then \(p(tu)=t/\lambda\) and
\(\sup_Rp\leq1/(\lambda q(u))<1-\varepsilon\).
Since \(s>\lambda\), one also has \(p(y)>1\).

For \textup{(ii)}, apply the construction in the proof of
Theorem~\ref{thm:bilateral}(ii) at \((1+\varepsilon)\lambda u\).
This point lies outside \(D\). It gives
\(q=P_{\alpha,\ell,\varphi}\), with \(0\leq\alpha<1\),
\(\varphi(u)=1\), and
\(-1/((1+\varepsilon)\lambda)<q(u)=(\alpha-1)/\ell<0\).
Also \(\sup_Dq\leq-1\).
The positive rescaling \(p=q/(-\lambda q(u))\) gives
\(p(tu)=-t/\lambda\) and
\(\sup_Dp\leq1/(\lambda q(u))<-1-\varepsilon\).
Since \(s<\lambda\), \(p(y)>-1\).
Positive rescaling changes only \(\ell\), preserving the parameter
regime and, when applicable, superlinearity.
\end{proof}

The preceding radial calibration allows us to characterize
exactly the levels at which strict separation can be achieved
with normalized separators.

\begin{corollary}\label{cor:radial-levels}
Let $\cF$ be a directionally fine family of 
closed cones in $X$, and let $\cP_{\cF}^{\nc}$ be its
associated norm-corrected family.
Let $R\subset X$ be closed and radiant,
$D\subset X$ closed and co-radiant, and $u\in S_X$.
Define the sets of admissible radial levels by
\begin{align*}
 \Lambda_R^{\cF}(u)
 &:=
 \{\lambda>0:\exists p\in\cP_{\cF}^{\nc},\ p(\lambda u)=1,\ \sup_Rp<1\},\\
 \Lambda_D^{\cF}(u)
 &:=
 \{\lambda>0:\exists p\in\cP_{\cF}^{\nc},\ p(\lambda u)=-1,\ \sup_Dp<-1\}.
\end{align*}
Then
\begin{equation}\label{eq:admissible-radial-levels}
 \Lambda_R^{\cF}(u)=(\sigma_R(u),+\infty),
 \qquad
 \Lambda_D^{\cF}(u)=(0,\rho_D(u)).
\end{equation}
Consequently,
\[
 \sigma_R(u)=\inf\Lambda_R^{\cF}(u),
 \qquad
 \rho_D(u)=\sup\Lambda_D^{\cF}(u),
\]
with $\inf\varnothing=+\infty$.

The first identity remains valid when the
separators are restricted to
$p=P_{\alpha,\ell,\varphi}$ with any prescribed
$\alpha>1$. The second identity remains valid
when they are restricted to
$p=P_{\alpha,\ell,\varphi}$ with $0\leq\alpha<1$.
\end{corollary}

\begin{proof}
For the radiant identity, if \(0<\lambda\leq\sigma_R(u)\),
closedness and radiancy give \(\lambda u\in R\);
the defining conditions of \(\Lambda_R^{\cF}(u)\) are impossible.
Conversely, for \(\lambda>\sigma_R(u)\), choose \(0<\varepsilon<1\)
with \(\sigma_R(u)<(1-\varepsilon)\lambda\) and take any \(s>\lambda\).
Corollary~\ref{cor:radial-calibration}(i) supplies the required
separator, with any prescribed \(\alpha>1\).

For the co-radiant identity, \(\lambda\geq\rho_D(u)\) implies
\(\lambda u\in D\), again excluding an admissible separator.
If \(0<\lambda<\rho_D(u)\), choose \(\varepsilon>0\) with
\((1+\varepsilon)\lambda<\rho_D(u)\) and any \(0<s<\lambda\).
Part~\textup{(ii)} of the same corollary gives a separator with
\(0\leq\alpha<1\).
Taking the endpoints of the two intervals proves the remaining claims.
\end{proof}

We recall the Minkowski gauge and cogauge,
which are standard tools in the study of radiant
and co-radiant sets; see
\cite[Definitions~2.3 and~2.5]{Zaffaroni2007}
and \cite[Section~3]{Zaffaroni2022}.

\begin{definition}\label{def:radial-gauges}
For closed radiant $R$ and closed co-radiant $D$,
their Minkowski gauge and cogauge are
\[
 \mu_R(x):=\inf\{t>0:x\in tR\},\qquad
 \nu_D(x):=\sup\{t>0:x\in tD\},
\]
with $\inf\varnothing=+\infty$ and
$\sup\varnothing=0$.
\end{definition}

Under our standing assumptions, both functions
vanish at $0_X$. Their values along each ray
are determined by the radial profiles:
\[
 \mu_R(tu)=\frac{t}{\sigma_R(u)},\qquad
 \nu_D(tu)=\frac{t}{\rho_D(u)},
 \qquad \forall t>0,\ \forall u\in S_X,
\]
where $1/0=+\infty$ and $1/(+\infty)=0$.

In particular,  $\rho_D$ is the reciprocal of the cogauge
on $S_X$.

\medskip

The characterization of admissible radial levels
in Corollary~\ref{cor:radial-levels} yields the following
representations of the Minkowski gauge and cogauge
in terms of strict norm-corrected separators.

\begin{corollary}\label{cor:strict-gauge-representations}
Let $\cF$ be a directionally fine family of
closed cones in $X$, and let $\cP_{\cF}^{\nc}$
be its associated norm-corrected family.
Let $R$ be closed and radiant and $D$ closed and
co-radiant.
Define
\begin{align*}
 \mathcal S_R^{\cF}
 &:=\{p\in\cP_{\cF}^{\nc}:\sup_Rp<1\},\\
 \mathcal S_D^{\cF}
 &:=\{p\in\cP_{\cF}^{\nc}:\sup_Dp<-1,\ p=P_{\alpha,\ell,\varphi}
 \text{ for some }0\leq\alpha<1\}.
\end{align*}
Then, for every $x\in X$,
\begin{equation}\label{eq:strict-gauge-representations}
 \mu_R(x)=\sup_{p\in\mathcal S_R^{\cF}}p(x),
 \qquad
 \nu_D(x)=-\sup_{p\in\mathcal S_D^{\cF}}p(x).
\end{equation}
\end{corollary}

\begin{proof}
For the gauge identity, if \(p\in\mathcal S_R^{\cF}\) and
\(x\in aR\), \(a>0\), then \(p(x)\leq a\).
Taking the infimum gives \(p\leq\mu_R\).
For \(x=tu\ne0_X\), Corollary~\ref{cor:radial-levels} supplies
the values \(p(x)=t/\lambda\) for every \(\lambda>\sigma_R(u)\).
Letting \(\lambda\downarrow\sigma_R(u)\) proves equality when
\(\sigma_R(u)<+\infty\), including \(\sigma_R(u)=0\).
If \(\sigma_R(u)=+\infty\), then \(\mu_R(x)=0\);
the functions \(-\|\cdot\|/\ell\in\mathcal S_R^{\cF}\),
as \(\ell\to+\infty\), give supremum zero.

For the cogauge identity, each \(p\in\mathcal S_D^{\cF}\)
is nonpositive by Proposition~\ref{prop:regimes}.
If \(\rho_D(u)<+\infty\), then
\(p(\rho_D(u)u)<-1\), whence \(p(tu)\leq-t/\rho_D(u)\).
If \(\rho_D(u)=+\infty\), nonpositivity gives the same bound,
with right side zero. Thus \(p\leq-\nu_D\).
Corollary~\ref{cor:radial-levels} supplies
\(p(tu)=-t/\lambda\) for every \(0<\lambda<\rho_D(u)\).
Letting \(\lambda\uparrow\rho_D(u)\), including an infinite endpoint,
proves equality. At zero, all functions vanish.
\end{proof}

We now specialize the norm-corrected
construction to the axial family, obtaining
explicit formulas for the corresponding
functions and cones.

\begin{definition}\label{def:ncuad}
For a norming pair $(u,f)$, $\delta>0$,
$\alpha\geq0$, and $\ell>0$,
the \emph{norm-corrected axial function}
is defined by
\begin{equation}\label{eq:ncuad}
 P_{\alpha,\ell,u,f,\delta}(x)
 :=
 \frac{\alpha\varphi_{u,f,\delta}(x)-\|x\|}{\ell}
 =
 \frac{\alpha f(x)-(\alpha/\delta)\|Q_{u,f}x\|-\|x\|}{\ell},
 \qquad \forall x\in X.
\end{equation}
The family of norm-corrected axial functions is
\[
 \cP_{\AX}^{\nc}
 :=
 \left\{
 P_{\alpha,\ell,u,f,\delta}:
 (u,f)\text{ is a norming pair},
 \ \alpha\geq0,\ \ell>0,\ \delta>0
 \right\}.
\]
For every $\alpha\geq0$, the associated
\emph{norm-corrected axial cone} is defined by
\[
 C^{\nc}_{\alpha,u,f,\delta}
 :=
 \{x\in X:\alpha\varphi_{u,f,\delta}(x)-\|x\|\geq0\}
 .
\]
\end{definition}

Since $\ell>0$, we have $ C^{\nc}_{\alpha,u,f,\delta}
 =\{P_{\alpha,\ell,u,f,\delta}\geq0\}$,  so the norm-corrected axial cone is independent
of $\ell$.

Norm-corrected axial functions are Lipschitz and superlinear,
and the parameter regimes of Proposition~\ref{prop:regimes} apply.
The following remark describes the geometry
of the associated cones in the three parameter regimes.

\begin{remark}\label{rem:corrected-axial-cones}
The cones \(C^{\nc}_{\alpha,u,f,\delta}\) are closed and
convex for all \(\alpha\geq0\).
For \(0\leq\alpha<1\), they reduce to \(\{0_X\}\);
the negative superlevels of the associated functions
give the bounded convex radiant neighbourhoods used
for co-radiant separation.

For \(\alpha=1\), the inequalities
\(\varphi_{u,f,\delta}\leq f\leq\|\cdot\|\)
force \(Q_{u,f}x=0_X\) and \(f(x)\geq0\)
on the zero superlevel, so
\(C^{\nc}_{1,u,f,\delta}=\Rp u\).

For \(\alpha>1\), the defining numerator equals
\(\alpha-1>0\) at \(u\),
so
\(u\in\operatorname{int}C^{\nc}_{\alpha,u,f,\delta}\).
\end{remark}

We now specialize the positive-level
separation result to the norm-corrected axial
family, obtaining explicit separating functions
and parameter bounds.

\begin{corollary}\label{cor:uad-corrected-positive}
Let $\cF_{\AX}$ be the convex directionally
fine family of axial cones in $X$, with associated
axial functions $\cG_{\AX}$ and norm-corrected
axial functions $\cP_{\AX}^{\nc}$. Let \(R\) be closed and radiant, \(y=su\notin R\),
with \(s>0\) and \(u\in S_X\). Choose
\[
 0<\theta<\min\{\dist(y,R),s\},\qquad
 s-\theta<\lambda<s,\qquad
 0<\delta<\min\left\{
 1,\frac{\theta-s+\lambda}{2\lambda}
 \right\}.
\]
For any norming pair \((u,f)\) and
any \(\alpha>1\), set
\(\ell:=(\alpha-1)\lambda\) and
\(P:=P_{\alpha,\ell,u,f,\delta}
\in\cP_{\AX}^{\nc}\).
Then
\[
 y\in\{P>1\}
 \subset\{P\geq1\}
 \subset X\setminus R.
\]
The strict and weak superlevels are
co-radiant neighbourhoods of \(y\).
\end{corollary}

\begin{proof}
Put \(H=(\alpha\varphi_{u,f,\delta}-\|\cdot\|)/(\alpha-1)\).
Proposition~\ref{prop:regimes} shows that \(H\) is adapted to a
cone contained in \(C_{u,f,\delta}\subset\cone B(u,\delta)\).
Also \(H(y)=s>\lambda\).
If \(H(x)\geq\lambda\), adaptation gives \(\|x\|\geq\lambda\).
For \(v=x/\|x\|\), Lemma~\ref{lem:directions} yields
\(\|v-u\|\leq2\delta\), so
\(\|\lambda v-y\|\leq2\lambda\delta+s-\lambda<\theta\).
Hence \(x\in\sdw B(y,\theta)\), which avoids \(R\) by radiancy.
Since \(P=H/\lambda\), this proves all inclusions.
Continuity and homogeneity give the neighbourhood assertions.
\end{proof}
Unlike the positive-level construction, the following co-radiant separation result does not require a small axial aperture. Instead, it uses the distance from the segment \([0_X,y]\) to the closed co-radiant set \(D\).

\begin{corollary}\label{cor:uad-bilateral}
Let $\cF_{\AX}$ be the convex directionally
fine family of axial cones in $X$, with associated
axial functions $\cG_{\AX}$ and norm-corrected
axial functions $\cP_{\AX}^{\nc}$.  Let $D\subset X$ be closed and co-radiant,
and let $y\notin D$.
If $y\neq0_X$, write $y=su$, with $s>0$
and $u\in S_X$, and put
\[
d_0:=\dist([0_X,y],D),\qquad
h:=d_0/4,\qquad r:=s+h.
\]
For any $\delta>0$ and any norming
pair $(u,f)$, set
\[
\ell:=\frac{h}{2(1+\delta)},\qquad
\alpha:=1-\frac{\ell}{r},\qquad
P:=P_{\alpha,\ell,u,f,\delta}.
\]
Then $0<\alpha<1$ and
$P\in\cP_{\AX}^{\nc}$.

If $y=0_X$, choose $0<\ell<d_D$
and define
\[
P(x):=-\frac{\|x\|}{\ell},
\qquad \forall x\in X.
\]
This function also belongs to $\cP_{\AX}^{\nc}$,
by taking $\alpha=0$.

In either case, the resulting
function $P$ satisfies
\[
y\in\{P>-1\}
\subset\{P\geq-1\}
\subset X\setminus D.
\]

The strict superlevel is open,
and the weak superlevel is closed.
Both are bounded convex radiant neighbourhoods
of $0_X$ and $y$.
\end{corollary}

\begin{proof}
If $y=0_X$, choose $0<\ell<d_D$ and
$P=-\|\cdot\|/\ell$. Then
$\{P\geq-1\}=\ell B_X\subset X\setminus D$,
and $P(0_X)=0>-1$.

Suppose now that $y=su\neq0_X$.
Co-radiancy gives $[0_X,y]\cap D=\varnothing$,
and compactness of the segment and closedness of
$D$ yield $d_0>0$. The chosen parameters satisfy
$1/2<\alpha<1$ and $P(y)=-s/r>-1$.

Let $P(x)\geq-1$, and put $t=f(x)$ and
$b=\|Q_{u,f}x\|$.
Proposition~\ref{prop:regimes} gives
$\|x\|\leq r=s+h$. Moreover,
$t\geq-\ell/\alpha$ and
$(\alpha/\delta)b\leq\ell+\alpha t-\|x\|
\leq\ell$, since $\alpha<1$ and
$t\leq\|x\|$.
Consequently,
$-h/(1+\delta)\leq t\leq s+h$ and
$b\leq\delta h/(1+\delta)<h$.

Set $a=\min\{s,\max\{0,t\}\}$.
Then $au\in[0_X,y]$, $|t-a|\leq h$, and
$\|x-au\|\leq b+|t-a|<2h<d_0$.
Thus $x\notin D$, proving
\[
y\in\{P>-1\}\subset\{P\geq-1\}
\subset X\setminus D.
\]
In both cases, continuity, positive homogeneity,
and superlinearity give the stated topological,
radiant, and convexity properties. Boundedness
follows from Proposition~\ref{prop:regimes}.
\end{proof}

\begin{remark}\label{rem:axial-shape}
The axial extent of the superlevels at level
$-1$ can be computed explicitly.
Let $(u,f)$ be a norming pair, let
$\delta,\ell>0$, $0\leq\alpha<1$, and set
$P:=P_{\alpha,\ell,u,f,\delta}$ as in
\eqref{eq:ncuad}.
Since $\varphi_{u,f,\delta}(tu)=t$
for every $t\in\R$,
\[
P(tu)=\frac{\alpha t-|t|}{\ell},
\qquad
\{tu:t\in\R,\ P(tu)\geq-1\}
=
\left[-\frac{\ell}{1+\alpha},
      \frac{\ell}{1-\alpha}\right]u.
\]
Thus the superlevel extends a distance
$\ell/(1+\alpha)$
in the direction $-u$ and a distance
$\ell/(1-\alpha)$
in the direction $u$.
For fixed $\ell$, the positive axial
extent tends to $+\infty$ as $\alpha\uparrow1$.
\end{remark}

The following example illustrates this geometry
in the Euclidean plane.

\begin{example}
\label{ex:planar-ncuad-cap}
In Euclidean \(\R^2\), let \(D=\{x:\|x\|_2\geq1\}\) and
\(y=(0,1/2)\). Choose the norming pair \((e_2,f_2)\),
\(f_2(x)=x_2\), and parameters
\(\alpha=3/5\), \(\ell=7/25\), \(\delta=2/5\).
Then \(\varphi_{e_2,f_2,\delta}(x)=x_2-(5/2)|x_1|\), and
\(P=P_{3/5,7/25,e_2,f_2,2/5}\) has weak superlevel
\begin{equation}\label{eq:planar-corrected}
 \{P\geq-1\}
 =\left\{x:\|x\|_2+\frac32|x_1|-\frac35x_2\leq\frac7{25}\right\}.
\end{equation}
Since \(P(y)=-5/7>-1\) and
\(P(x)\leq-((1-\alpha)/\ell)\|x\|_2\),
\[
 y\in\{P>-1\}\subset\{P\geq-1\}
 \subset\frac7{10}B_X\subset X\setminus D.
\]
The cap is a bounded convex radiant neighbourhood of zero and \(y\).
\end{example}

\section{Evaluation polarity and primal bipolar identities}\label{sec:polarity}

In this section, we introduce functional classes generated by adapted generators and their norm corrections, and define the corresponding evaluation polarities. We establish bipolar identities that recover closed radiant and co-radiant hulls, derive formulas for their radial profiles, and obtain explicit axial versions of these results.

\medskip

We introduce three nested classes of
positively homogeneous functions. Starting from
the selected adapted generators, we consider
their nonnegative scalar multiples and the
closure of their finite nonnegative sums.
Allowing the additional generator
$-\|\cdot\|$ yields a larger class containing
the norm-corrected functions introduced in
Section~\ref{sec:bilateral}.

\begin{definition}\label{def:function-classes}
Let \(\cF\) be a directionally fine family of closed cones in
\(X\), with selected Lipschitz adapted generators \(\cG_{\cF}\) as in
Definition~\ref{def:selected-generators}. Let \(\cH(X)\) be the space of
continuous positively homogeneous real functions, with
\(\|p\|_{\cH}=\sup_{u\in S_X}|p(u)|\). Define
\begin{align}
 \mathcal A_{\cF}&=\cone(\cG_{\cF}),\notag\\
 \cK_{\cF}&=\cl_{\|\cdot\|_{\cH}}\co(\cone(\cG_{\cF}))
 =\cl_{\|\cdot\|_{\cH}}
 \left\{\sum_{j=1}^m a_j\varphi_j:
 m\in\mathbb N,\ a_j\geq0,\ \varphi_j\in\cG_{\cF}\right\},\notag\\
 \cM_{\cF}&=\cl_{\|\cdot\|_{\cH}}
 \co(\cone(\cG_{\cF}\cup\{-\|\cdot\|\}))\notag\\
 &=\cl_{\|\cdot\|_{\cH}}
 \left\{\sum_{j=1}^m a_j\varphi_j-b\|\cdot\|:
 m\in\mathbb N,\ a_j,b\geq0,\ \varphi_j\in\cG_{\cF}\right\}.
 \label{eq:classes}
\end{align}
All closures are taken with respect to \(\|\cdot\|_{\cH}\),
equivalently, to uniform convergence on \(S_X\). 
Notice that $\mathcal A_{\cF}$ consists of
nonnegative scalar multiples of individual
generators and need not be convex or closed,
whereas $\cK_{\cF}$ and $\cM_{\cF}$ are
closed convex cones.

\end{definition}

Continuity at zero gives $r>0$ with $|p(x)|\leq1$
for $\|x\|\leq r$; homogeneity then gives
$\|p\|_{\cH}\leq1/r$.
Evaluation $e_x(p)=p(x)$ is linear and continuous,
with $|p(x)|\leq\|p\|_{\cH}\|x\|$.

If every generator in $\cG_{\cF}$ is
superlinear, then every function in
$\mathcal A_{\cF}$, $\cK_{\cF}$, and $\cM_{\cF}$
is superlinear, since $-\|\cdot\|$ is
superlinear and superlinearity is preserved
under nonnegative linear combinations
and uniform limits on $S_X$.

The following inclusion links the
norm-corrected functions of Section~\ref{sec:bilateral}
with the generated functional classes and will
be used in the subsequent bipolar identities.

\begin{proposition}\label{prop:norm-corrected-inclusion}

Let $\cF$ be a directionally fine family
of closed cones in $X$, with selected
Lipschitz adapted generators $\cG_{\cF}$.
Let $\cP_{\cF}^{\nc}$ and $\cM_{\cF}$ be the
associated functional classes defined in
Definitions~\ref{def:norm-corrected}
and~\ref{def:function-classes}, respectively.
Then
\[
\cP_{\cF}^{\nc}
\subset
\co(\cone(\cG_{\cF}\cup\{-\|\cdot\|\}))
\subset\cM_{\cF}.
\]
\end{proposition}

\begin{proof}
For every member of the norm-corrected family,
\(P_{\alpha,\ell,\varphi}=(\alpha/\ell)\varphi+
(1/\ell)(-\|\cdot\|)\).
Both coefficients are nonnegative, so this function belongs to the
convex conic hull of the enlarged generator family.
Taking its norm closure gives the second inclusion.
\end{proof}

The following example shows that enlarging
$\cK_{\cF}$ by the negative norm generator is
essential for negative-level separation, even
in dimension one.

\begin{remark}\label{rem:conic-hull-missing-sign}

In $X=\R$, consider the convex
directionally fine family
$\cF=\{\R_+,-\R_+\}$, with selected
adapted generators $\varphi_+(x)=x$ and
$\varphi_-(x)=-x$.
Then $\cK_{\cF}$ consists precisely of
the linear functions $p(x)=ax$, $a\in\R$,
so $p(1)+p(-1)=0$ for every
$p\in\cK_{\cF}$.

The inequalities $p(1)\leq-1$ and
$p(-1)\leq-1$ would give $0\leq-2$.

Consequently, $\cK_{\cF}$ cannot
separate $0$ from the closed co-radiant set
$D:=\{x\in\R:|x|\geq1\}$ at level $-1$.
By contrast, for $0<\ell<1$, the
norm-corrected function
$P(x):=-|x|/\ell\in\cP_{\cF}^{\nc}
\subset\cM_{\cF}$ satisfies
$P(0)=0>-1>\sup_{d\in D}P(d)=-1/\ell$.
\end{remark}

We now introduce two evaluation polarities,
at levels $1$ and $-1$. The separation results
of Sections~\ref{sec:positive} and~\ref{sec:bilateral} will allow us to identify
their primal bipolars with the corresponding
closed radiant and co-radiant hulls.

\begin{definition}\label{def:polars}
Let $\mathcal B\subset\cH(X)$, $E\subset X$,
and $P\subset\mathcal B$.
The \emph{positive} and \emph{negative
evaluation polars} are defined by

\begin{align*}
 E^{\oplus_{\mathcal B}}
 &=\{p\in\mathcal B:p(e)\leq1\ (e\in E)\},&
 P^{\oplus_X}
 &=\{x\in X:p(x)\leq1\ (p\in P)\},\\
 E^{\ominus_{\mathcal B}}
 &=\{p\in\mathcal B:p(e)\leq-1\ (e\in E)\},&
 P^{\ominus_X}
 &=\{x\in X:p(x)\leq-1\ (p\in P)\}.
\end{align*}

\end{definition}

Both pairs of polar operators reverse
inclusions. Moreover, they satisfy
\[
E\subset P^{\oplus_X}
\iff P\subset E^{\oplus_{\mathcal B}},
\qquad
E\subset P^{\ominus_X}
\iff P\subset E^{\ominus_{\mathcal B}}.
\]

The following theorem identifies the
primal bipolars associated with the prescribed
functional classes.

\begin{theorem}\label{thm:bipolar}

Let $\cF$ be a directionally fine family
of closed cones in $X$, with associated
functional classes $\mathcal A_{\cF}$,
$\cK_{\cF}$, $\cM_{\cF}$, and
$\cP_{\cF}^{\nc}$. Then the following
assertions hold.
\begin{enumerate}
\item Every nonempty $E\subset X$ satisfies
\begin{equation}\label{eq:positive-bipolar}
 (E^{\oplus_{\mathcal B}})^{\oplus_X}
 =\cl(\rad E)
 \quad\text{for }
 \mathcal B\in
 \{\mathcal A_{\cF},\cK_{\cF},\cM_{\cF}\}.
\end{equation}
\item If $E\ne\varnothing$ and
$\dist(0_X,E)>0$, then
\begin{equation}\label{eq:negative-bipolar}
 (E^{\ominus_{\mathcal B}})^{\ominus_X}
 =\cl(\sdw E)
 \quad\text{for }
 \mathcal B\in
 \{\cP_{\cF}^{\nc},\cM_{\cF}\}.
\end{equation}
For $\cP_{\cF}^{\nc}$, the restriction
$0\leq\alpha<1$ suffices.
\end{enumerate}

In particular, the positive bipolars
fix every nonempty closed radiant set,
whereas the negative bipolars fix every
nonempty closed co-radiant set $D$ with
$0_X\notin D$.
\end{theorem}

\begin{proof}
For \textup{(i)}, each sublevel \(\{p\leq1\}\),
\(p\in\cH(X)\), is closed and radiant.
If it contains \(E\), it therefore contains \(R=\cl(\rad E)\).
This proves the hull inclusion for all three classes.
If \(R\ne X\), Theorem~\ref{thm:radiant-point} separates each
\(y\notin R\) by a member of \(\mathcal A_{\cF}\).
That member belongs to the polar of \(E\) and excludes \(y\)
from each bipolar, since
\(\mathcal A_{\cF}\subset\cK_{\cF}\subset\cM_{\cF}\).
If \(R=X\), the hull inclusion already proves equality.

For \textup{(ii)}, each sublevel \(\{p\leq-1\}\) is closed
and co-radiant, so every negative bipolar contains \(D=\cl(\sdw E)\).
Moreover \(\dist(0_X,D)\geq\dist(0_X,E)>0\).
Theorem~\ref{thm:bilateral}(ii) separates every \(y\notin D\)
by a norm-corrected member with \(0\leq\alpha<1\).
It belongs to the negative polar of \(E\) and excludes \(y\)
from the bipolar. The inclusion in
Proposition~\ref{prop:norm-corrected-inclusion} gives the same
conclusion for \(\cM_{\cF}\).
For closed radial sets, the corresponding hull equals the set itself.
\end{proof}

The bipolar identities also admit
explicit representations as intersections
of evaluation sublevel sets.

\begin{remark}\label{rem:polar-intersections}

Let $R$ be a nonempty closed radiant
set and $D$ a nonempty closed co-radiant
set with $0_X\notin D$.
Theorem~\ref{thm:bipolar} yields
\[
 R=\bigcap_{p\in R^{\oplus_{\mathcal B}}}
 \{x\in X:p(x)\leq1\},\qquad
 D=\bigcap_{p\in D^{\ominus_{\mathcal B}}}
 \{x\in X:p(x)\leq-1\},
\]
using the classes $\mathcal B$ in
\eqref{eq:positive-bipolar} for $R$
and those in \eqref{eq:negative-bipolar}
for $D$.

These representations recover
both sets without requiring convexity.
\end{remark}

The bipolar identities also yield
representations of the radial profiles
in terms of the values of polar functions
along unit directions.

\begin{corollary}\label{cor:polar-radial-profiles}

Let $\cF$ be a directionally fine family
of closed cones in $X$.
Let $R\subset X$ be a nonempty closed radiant
set and $D\subset X$ a nonempty closed
co-radiant set with $0_X\notin D$.
For $u\in S_X$ and
$\mathcal B\in
\{\mathcal A_{\cF},\cK_{\cF},\cM_{\cF}\}$,
\begin{equation}\label{eq:polar-radiant-profile}
 \sigma_R(u)=
 \inf_{\substack{p\in R^{\oplus_{\mathcal B}}\\p(u)>0}}
 \frac1{p(u)},
 \qquad \inf\varnothing=+\infty.
\end{equation}
If $\rho_D(u)<+\infty$, then, for
$\mathcal B\in\{\cP_{\cF}^{\nc},\cM_{\cF}\}$,
\begin{equation}\label{eq:polar-coradiant-profile}
 \rho_D(u)=-\frac1{\displaystyle
 \sup_{p\in D^{\ominus_{\mathcal B}}}p(u)}.
\end{equation}
\end{corollary}

\begin{proof}
For the radiant formula, Theorem~\ref{thm:bipolar} gives
\(tu\in R\), \(t\geq0\), exactly when \(tp(u)\leq1\)
for every \(p\in R^{\oplus_{\mathcal B}}\).
Values \(p(u)\leq0\) impose no condition; positive values impose
\(t\leq1/p(u)\).
Their intersection is the interval with endpoint
\eqref{eq:polar-radiant-profile}, with endpoint \(+\infty\)
when there are no positive values.

For the co-radiant formula, attainment of \(\rho_D(u)\) gives
\(p(u)\leq-1/\rho_D(u)<0\) for every
\(p\in D^{\ominus_{\mathcal B}}\).
This polar is nonempty, for example because a sufficiently large
negative multiple of the norm belongs to both admissible classes.
The bipolar identity gives \(tu\in D\) exactly when
\(t\geq-1/p(u)\) for every such \(p\).
The function \(s\mapsto-1/s\) is continuous and increasing on
\((-\infty,0)\). Since the supremum of the values \(p(u)\) is
strictly negative, the supremum of their thresholds equals
\(-1/\sup_p p(u)\), proving \eqref{eq:polar-coradiant-profile}.
\end{proof}

We now introduce the functional classes
generated by the axial family $\cG_{\AX}$.
The explicit axial separators of Section~\ref{sec:bilateral}
will yield the corresponding bipolar identities.

\begin{definition}\label{def:axial-function-classes}
The \emph{generated axial functional classes}
are defined by
\[
 \cK_{\AX}:=
 \cl_{\|\cdot\|_{\cH}}
 \co\bigl(\cone(\cG_{\AX})\bigr),
 \qquad
 \cM_{\AX}:=
 \cl_{\|\cdot\|_{\cH}}
 \co\bigl(\cone(\cG_{\AX}
 \cup\{-\|\cdot\|\})\bigr).
\]
\end{definition}

Since
$P_{\alpha,\ell,u,f,\delta}
=(\alpha/\ell)\varphi_{u,f,\delta}
+(1/\ell)(-\|\cdot\|)$,
the norm-corrected axial family satisfies
\[
 \cP_{\AX}^{\nc}
 \subset\cM_{\AX}\subset\cH(X).
\]

The axial separation results of
Corollaries~\ref{cor:uad-corrected-positive}
and~\ref{cor:uad-bilateral} yield the following
bipolar identities.

\begin{corollary}\label{cor:ncuad-bipolar}
Let $\cF_{\AX}$ be the convex directionally
fine family of axial cones in $X$, with associated
norm-corrected axial family $\cP_{\AX}^{\nc}$
and generated functional class $\cM_{\AX}$.

Every nonempty $E\subset X$ satisfies
\begin{equation}\label{eq:ncuad-positive-bipolar}
 \bigl(E^{\oplus_{\cP_{\AX}^{\nc}}}\bigr)^{\oplus_X}
 =\cl(\rad E).
\end{equation}
For every fixed $\alpha>1$, this identity remains
valid if the functional class $\cP_{\AX}^{\nc}$
in the first polar is restricted to the functions
$P_{\alpha,\ell,u,f,\delta}$.

If $\dist(0_X,E)>0$, then
\begin{equation}\label{eq:ncuad-negative-bipolar}
 \bigl(E^{\ominus_{\cP_{\AX}^{\nc}}}\bigr)^{\ominus_X}
 =\cl(\sdw E).
\end{equation}
This identity remains valid if the functional
class $\cP_{\AX}^{\nc}$ in the first polar is
restricted to the functions
$P_{\alpha,\ell,u,f,\delta}$ with
$0\leq\alpha<1$.

Both identities remain valid with $\cM_{\AX}$
in place of $\cP_{\AX}^{\nc}$.

The first identity fixes every nonempty
closed radiant set, whereas the second fixes
every nonempty closed co-radiant set
that does not contain $0_X$.
\end{corollary}

\begin{proof}
For the positive identity, sublevels at \(1\) of functions in
\(\cH(X)\) contain the closed radiant hull whenever they contain \(E\).
Corollary~\ref{cor:uad-corrected-positive} excludes each exterior
point using any prescribed \(\alpha>1\).
This proves the identity for the full and restricted axial classes.

For the negative identity, \(\dist(0_X,E)>0\) ensures that
\(\cl(\sdw E)\) avoids zero.
Corollary~\ref{cor:uad-bilateral} excludes each exterior point
by a separator with \(0\leq\alpha<1\), while negative sublevels
give the opposite inclusion.
Finally, \(\cP_{\AX}^{\nc}\subset\cM_{\AX}\subset\cH(X)\)
preserves the hull inclusions and all separating functions,
proving both identities for \(\cM_{\AX}\).
\end{proof}

\section{Relative approximation of co-radiant sets}\label{sec:approximation}

For closed co-radiant sets with compact
direction sets, directional fineness allows
finite relative outer approximations by
positive superlevels of selected adapted
generators. In this section, we establish directional and
bounded-section error estimates, identify
conditions ensuring preservation of separation,
and obtain an explicit axial recovery map.

\begin{definition}\label{def:enlargements}

Let $D\subset X$ be a nonempty closed
co-radiant set with $0_X\notin D$, and let
$0<\eta<1$. The \emph{co-radiant
$\eta$-neighbourhood} of $D$ is defined by
\[
D_\eta
:=\bigcup_{a\in\Gamma_D}
\sdw B(a,\eta\|a\|).
\]

\end{definition}

The set $D_\eta$ is co-radiant and
satisfies $D\subset\intt D_\eta$.
No closedness of $D_\eta$ is asserted.

\begin{definition}\label{def:relative-outer-approx}

Let $D\subset X$ be a nonempty closed
co-radiant set with $0_X\notin D$,
let $\cF$ be a directionally fine family
of closed cones in $X$, with selected
adapted generators $\cG_{\cF}$, and let
$0<\eta<1$.

A \emph{finite relative $\eta$-outer approximation}
of $D$ associated with $\cF$ is a set
\[
D^\eta:=\bigcup_{i=1}^{m}G_i,
\qquad
G_i:=\{\varphi_i\geq\lambda_i\},
\]

where $m\in\mathbb N$,
$\varphi_i\in\cG_{\cF}$ and $\lambda_i>0$
for $1\leq i\leq m$, such that
\[
D\subset\intt D^\eta
\subset D^\eta\subset D_\eta.
\]
\end{definition}

We now establish finite relative outer approximations
of directionally compact co-radiant sets and derive
quantitative geometric error estimates.
The finite selection takes place on $\Sigma_D$,
rather than on the possibly unbounded radial base
$\Gamma_D$; the required relatively open cover
is obtained using lower semicontinuity of $\rho_D$
and continuity of the selected generators.

\begin{theorem}\label{thm:relative-approx}

Let $D\subset X$ be a nonempty closed,
co-radiant, and directionally compact set with
$0_X\notin D$, and let $\cF$ be a directionally
fine family of closed cones in $X$.

For every $0<\eta<1$, there exist
$m\in\mathbb N$, points $a_i\in\Gamma_D$,
localized pairs $(C_i,u_i)\in\mathcal L_{\cF}$
with $u_i=a_i/\|a_i\|$, selected generators
$\varphi_i=\varphi_{C_i,u_i}\in\cG_{\cF}$,
and levels $0<\lambda_i<\|a_i\|$,
$1\leq i\leq m$.

Set $G_i:=\{\varphi_i\geq\lambda_i\}$
and $D^\eta:=\bigcup_{i=1}^m G_i$.

Each $a_i$ belongs to the strict
superlevel $\{\varphi_i>\lambda_i\}$, and
$G_i\subset\sdw B(a_i,\eta\|a_i\|)$.

Moreover, $D^\eta$ is closed and
co-radiant, and
\begin{equation}\label{eq:relative-inclusion}
D\subset\bigcup_{i=1}^m
\{\varphi_i>\lambda_i\}
\subset\intt D^\eta
\subset D^\eta\subset D_\eta.
\end{equation}

For every $i$ and $x\in G_i$,
there exists $t\geq1$ such that
$d_x:=ta_i\in D$ and
\begin{equation}\label{eq:abstract-recovery}
\|x-d_x\|\leq
\frac{\eta}{1-\eta}\|x\|.
\end{equation}

Furthermore,
\begin{align}
d_H(\Sigma_{D^\eta},\Sigma_D)
&\leq\eta,
\label{eq:relative-angular}\\
e(D^\eta\cap TB_X,D)
&\leq\frac{T\eta}{1-\eta},
\qquad \forall T>0.
\notag
\end{align}

If $\mathcal I_D<+\infty$, then
\[
d_H(D^\eta\cap TB_X,D\cap TB_X)
\leq\frac{2T\eta}{1-\eta},
\qquad \forall T\geq\mathcal I_D.
\]

For convex directionally fine $\cF$,
the $\varphi_i$ may be chosen superlinear
and the $G_i$ convex.
\end{theorem}

\begin{proof}
For the finite selection, fix \(u\in\Sigma_D\), put
\(s_u=\rho_D(u)\), \(a_u=s_u u\), and choose
\[
 (1-\eta)s_u<\lambda_u<s_u,\qquad
 0<\theta_u<
 \frac{\lambda_u-(1-\eta)s_u}{2\lambda_u}<\frac{\eta}{2}.
\]
Directional fineness supplies a localized cone
\(C_u\subset\cone B(u,\theta_u)\) and its selected generator \(\varphi_u\).
Theorem~\ref{thm:local}, with radius \(\eta s_u\), gives the
required local sandwich at \(a_u\).
The directional sets
\begin{equation}\label{eq:direction-cover}
 W_u=\{v\in\Sigma_D:\rho_D(v)\varphi_u(v)>\lambda_u\}
\end{equation}
contain \(u\). They are relatively open: their points satisfy
\(\varphi_u>0\), and on that open set the product of the positive
lower semicontinuous function \(\rho_D\) and the positive continuous
function \(\varphi_u\) is lower semicontinuous.
Compactness gives a finite subcover \(W_{u_1},\ldots,W_{u_m}\).
Retain its local data \(a_i=\rho_D(u_i)u_i\), \(\lambda_i\),
\(\varphi_i\), \(C_i\), and \(\theta_i\).
For \(x=t\rho_D(v)v\in D\), \(t\geq1\), some \(i\) satisfies
\(\varphi_i(x)=t\rho_D(v)\varphi_i(v)>\lambda_i\).
The strict superlevels are open and contained in \(D^\eta\);
the local sandwiches then prove \eqref{eq:relative-inclusion}.
The finite union is closed and co-radiant.

For recovery, write \(x\in G_i\) as \(x=t(a_i+h)\), with
\(t\geq1\), \(\|h\|\leq\eta\|a_i\|\).
Co-radiancy gives \(d_x=ta_i\in D\), and
\[
 \|x\|\geq t(1-\eta)\|a_i\|,\qquad
 \|x-d_x\|=t\|h\|\leq\frac{\eta}{1-\eta}\|x\|.
\]
Taking the supremum over \(D^\eta\cap TB_X\) proves the excess bound.
Each \(v\in\Sigma_{D^\eta}\) belongs to some \(C_i\cap S_X\);
hence \(\dist(v,\Sigma_D)\leq\|v-u_i\|\leq2\theta_i<\eta\).
The opposite angular excess is zero, proving \eqref{eq:relative-angular}.

For the truncated Hausdorff estimate, let \(T\geq\mathcal I_D\).
If \(\|d_x\|>T\), replace \(d_x\) by \(z=Tu_i\).
Since \(T\geq\rho_D(u_i)\), this point belongs to \(D\cap TB_X\).
Also \(\|d_x\|-T\leq\|d_x\|-\|x\|\leq\|d_x-x\|\),
so \(\|x-z\|\leq2\|x-d_x\|\).
If \(\|d_x\|\leq T\), no replacement is needed.
The opposite truncated excess is zero because \(D\subset D^\eta\).
Finally, convex directional fineness allows superlinear local generators,
whose positive superlevels are convex.
\end{proof}

\begin{remark}
\label{rem:upper-bound}\label{rem:approximation-scope}

The finite selection in
Theorem~\ref{thm:relative-approx} requires
compactness of $\Sigma_D$, but not of $\Gamma_D$.
The condition $T\geq\mathcal I_D$ in the
truncated Hausdorff estimate ensures that
radial truncation of recovered points
remains in $D$.

The bounded-section estimates cannot, in general, be extended to a
finite global Hausdorff distance, even for a finite relative outer
approximation. To see this, work in Euclidean $\R^2$ and set
$D=\{te_2:t\geq1\}$ and $\eta=3/4$. Then $D$ is closed and
co-radiant, $0_X\notin D$, and
$\Sigma_D=\{e_2\}$ and $\Gamma_D=\{e_2\}$.
For the norming pair $u=e_2$, $f(x)=x_2$, and $\delta=1/2$,
we have $Q_{u,f}x=(x_1,0)$ and
$\varphi_{u,f,\delta}(x)=x_2-2|x_1|$. Thus the single positive
superlevel
\[
E:=\{\varphi_{u,f,\delta}\geq1/2\}
  =\{x:x_2-2|x_1|\geq1/2\}
\]
is associated with the axial directionally fine family. It is closed
and co-radiant. Moreover, for $te_2\in D$ one has
$\varphi_{u,f,\delta}(te_2)=t\geq1>1/2$, so
$D\subset\intt E$.

We check the required outer inclusion using
$D_\eta=\sdw B(e_2,\eta)$ from
Definition~\ref{def:enlargements}. If $x=(x_1,x_2)\in E$,
then $x_2\geq1/2$ and $|x_1|\leq(x_2-1/2)/2$.
When $x_2\geq1$, write
\[
x=x_2\bigl(e_2+(x_1/x_2,0)\bigr),\qquad
\frac{|x_1|}{x_2}\leq\frac12-\frac{1}{4x_2}
<\frac12<\eta.
\]
Hence $x\in\sdw B(e_2,\eta)$. When $1/2\leq x_2\leq1$,
\[
\|x-e_2\|\leq |x_1|+1-x_2
\leq\frac{x_2-1/2}{2}+1-x_2
=\frac34-\frac{x_2}{2}\leq\frac12<\eta,
\]
so $x\in B(e_2,\eta)\subset D_\eta$ as well. Therefore
$D\subset\intt E\subset E\subset D_\eta$, and $E$ is a finite
relative $\eta$-outer approximation in the sense of
Definition~\ref{def:relative-outer-approx}.

Finally, for every $t\geq1$ the point
$x_t=(t/2,t+1/2)$ satisfies
$\varphi_{u,f,\delta}(x_t)=1/2$, hence $x_t\in E$.
Its orthogonal projection onto $D$ is $(0,t+1/2)$, so
\[
\dist(x_t,D)=t/2\longrightarrow+\infty.
\]
Consequently $e(E,D)=d_H(E,D)=+\infty$.
\end{remark}

The following lemma establishes a
positive relative separation gap between
a closed directionally compact co-radiant set
and a disjoint closed radiant set. This gap
will ensure that sufficiently fine relative
outer approximations preserve disjointness.

\begin{lemma}\label{lem:relative-base-gap}

Let $D\subset X$ be a nonempty closed,
co-radiant, and directionally compact set,
and let $R\subset X$ be a nonempty closed
radiant set such that $D\cap R=\varnothing$.
Then
\[
0<\gamma_{D,R}:=
\inf_{a\in\Gamma_D}
\frac{\dist(a,R)}{\|a\|}\leq1.
\]
\end{lemma}

\begin{proof}
Since \(0_X\in R\), one has \(\dist(a,R)\leq\|a\|\), proving
the upper bound. If the infimum were zero, choose
\(a_n=s_nu_n\in\Gamma_D\) and \(r_n\in R\) with
\(\|a_n-r_n\|/s_n\to0\).
Compactness gives a subsequence with \(u_n\to u\in\Sigma_D\).
If \(s_n\) has a bounded subsequence, take \(s_n\to s\geq d_D\).
Then \(a_n,r_n\to su\), and closedness gives \(su\in D\cap R\).

Otherwise take \(s_n\to+\infty\) and put \(b=\rho_D(u)<+\infty\).
For large \(n\), \(b/s_n\leq1\), so radiancy gives
\((b/s_n)r_n\in R\).
The bound
\(\|(b/s_n)r_n-bu\|\leq b\|r_n-a_n\|/s_n+b\|u_n-u\|\to0\)
and closedness imply \(bu\in R\).
Radial attainment gives \(bu\in D\), again a contradiction.
\end{proof}

The relative gap from
Lemma~\ref{lem:relative-base-gap} ensures
that sufficiently fine approximations
preserve separation.

\begin{corollary}\label{cor:preserving}

Let $D$ and $R$ satisfy the hypotheses
of Lemma~\ref{lem:relative-base-gap},
and let $\cF$ be a directionally fine
family of closed cones in $X$.
For every $0<\eta<\gamma_{D,R}$,
Theorem~\ref{thm:relative-approx}
provides a finite relative
$\eta$-outer approximation $D^\eta$
satisfying
\[
D\subset\intt D^\eta
\subset D^\eta\subset D_\eta
\subset X\setminus R.
\]

Moreover, the associated scalarization
$\Psi:=\max_{1\leq i\leq m}
\varphi_i/\lambda_i$ is Lipschitz
and positively homogeneous, satisfies
$D^\eta=\{\Psi\geq1\}$, and strictly
separates $D$ from $R$ in the sense that
$\Psi(d)>1>\Psi(r)$ for all
$d\in D$ and $r\in R$.
\end{corollary}

\begin{proof}
Suppose \(t(a+h)\in R\), where \(t\geq1\), \(a\in\Gamma_D\),
and \(\|h\|\leq\eta\|a\|\).
Radiancy gives \(a+h\in R\), so
\(\gamma_{D,R}\leq\dist(a,R)/\|a\|\leq\eta\), a contradiction.
Thus \(D_\eta\cap R=\varnothing\).
Apply Theorem~\ref{thm:relative-approx}.
Its strict cover gives \(\Psi(d)>1\) for \(d\in D\);
disjointness of every weak superlevel from \(R\) gives
\(\varphi_i(r)/\lambda_i<1\) for every \(i\), hence \(\Psi(r)<1\).
A finite maximum preserves positive homogeneity and Lipschitz continuity.
\end{proof}

For the axial family, finite relative
outer approximations admit an explicit recovery
map that improves the relative error bound
of Theorem~\ref{thm:relative-approx}
from $\eta/(1-\eta)$ to $\eta$.

\begin{theorem}\label{thm:uad-recovery}

Let $D\subset X$ be a nonempty closed,
co-radiant, and directionally compact set
with $0_X\notin D$. For every $0<\eta<1$,
there exist $m\in\mathbb N$, norming pairs
$(u_i,f_i)$ with $u_i\in\Sigma_D$,
and parameters $0<\delta_i\leq\eta/2$,
$1\leq i\leq m$, such that, setting
$a_i:=\rho_D(u_i)$,
$\lambda_i:=a_i/(1+\eta)$, and
$\varphi_i:=\varphi_{u_i,f_i,\delta_i}$,
the set
$D^\eta:=\bigcup_{i=1}^m
\{\varphi_i\geq\lambda_i\}$
satisfies
$D\subset\intt D^\eta
\subset D^\eta\subset D_\eta$.

For every $i$ and
$x\in\{\varphi_i\geq\lambda_i\}$,
define the recovery map
\begin{equation}\label{eq:uad-recovery-map}
d_i(x):=\max\{f_i(x),a_i\}u_i.
\end{equation}
Then
\begin{equation}\label{eq:uad-recovery-estimate}
d_i(x)\in D,\qquad
\|x-d_i(x)\|\leq\eta\|x\|,\qquad
\|d_i(x)\|\leq(1+\eta)\|x\|.
\end{equation}
Consequently,
\[
\begin{aligned}
d_H(\Sigma_{D^\eta},\Sigma_D)&\leq\eta,\\
e(D^\eta\cap TB_X,D)&\leq T\eta,
\qquad \forall T>0.
\end{aligned}
\]
If $\mathcal I_D<+\infty$, then
\[
d_H(D^\eta\cap TB_X,D\cap TB_X)
\leq2T\eta,
\qquad \forall T\geq\mathcal I_D.
\]
\end{theorem}

\begin{proof}
For each \(u\in\Sigma_D\), choose a norming functional,
\(0<\delta_u\leq\eta/2\), and \(\lambda_u=\rho_D(u)/(1+\eta)\).
The directional sets \eqref{eq:direction-cover} contain \(u\)
and are relatively open by the same lower semicontinuity argument
as in Theorem~\ref{thm:relative-approx}.
A finite subcover therefore gives a strict cover of \(D\).

For recovery from a selected piece, put \(t=f_i(x)\),
\(q=\|Q_{u_i,f_i}x\|\), and \(m=\max\{t,a_i\}\).
The level inequality gives \(t\geq\lambda_i>0\) and
\(q\leq\delta_i(t-\lambda_i)\).
Since \(m\geq a_i\), \(d_i(x)=mu_i\in D\).
If \(t\geq a_i\), then \(\|x-d_i(x)\|=q\leq\eta t\).
If \(t<a_i\), use \(a_i-t\leq a_i-\lambda_i=\eta\lambda_i\)
and \(\delta_i\leq\eta\) to obtain
\(\|x-d_i(x)\|\leq q+a_i-t
\leq\eta(t-\lambda_i)+\eta\lambda_i=\eta t\).
Moreover \(t\leq\|x\|\) and \(m\leq(1+\eta)t\).
These bounds prove \eqref{eq:uad-recovery-estimate}.
Writing \(x=(m/a_i)(a_i u_i+h)\), with
\(h=(a_i/m)(x-mu_i)\), gives \(\|h\|\leq\eta a_i\);
hence \(x\in D_\eta\).

For the angular estimate, \(x=tu_i+Q_{u_i,f_i}x\) and \(q<t\).
Equation~\eqref{eq:normalization} gives
\(\|x/\|x\|-u_i\|\leq2q/t
\leq2\delta_i(1-\lambda_i/t)\leq\eta\).
The opposite angular excess is zero.
The recovery estimate gives the bounded-section excess.
Finally, for \(T\geq\mathcal I_D\), replace \(d_i(x)\) by \(Tu_i\)
if necessary. This point remains feasible and the replacement
adds at most \(\|x-d_i(x)\|\), proving the truncated Hausdorff bound.
\end{proof}

In particular, if $R$ is closed and
radiant with $D\cap R=\varnothing$, then
the axial approximations of
Theorem~\ref{thm:uad-recovery} satisfy
$D^\eta\subset X\setminus R$ whenever
$0<\eta<\gamma_{D,R}$, by
Corollary~\ref{cor:preserving}.

\section{Penalization, recovery, and convergence}\label{sec:optimization}

In this section, we study the minimization of a Lipschitz
objective over a closed co-radiant set $D$
through finite relative outer approximations
$D^\eta$.
For superlinear generators, a directional
repair yields a residual error bound and
exact penalization of the finite approximations.
We then recover feasible points in $D$,
with separate estimates for the approximation
error and the optimization tolerance, and
establish subsequential convergence of bounded
approximate minimizers as both errors vanish.
Finally, we specialize the recovery estimates
to the axial family and illustrate them in
a three-ray production model.

We introduce a residual that measures
constraint violation for a finite union of
positive superlevels, rather than for the
original feasible set $D$.

\begin{definition}\label{def:residual}

Let $m\in\mathbb N$ and, for
$1\leq i\leq m$, let
$\varphi_i:X\to\R$ be continuous
superlinear functions, $u_i\in S_X$
satisfy $\varphi_i(u_i)=1$,
and $\lambda_i>0$. Define
\begin{equation}\label{eq:residual}
\begin{gathered}
E:=\bigcup_{i=1}^m\{\varphi_i\geq\lambda_i\},
\qquad
F(x):=\max_{1\leq i\leq m}
\{\varphi_i(x)-\lambda_i\},\\
r_E(x):=\pos{-F(x)}
=\min_{1\leq i\leq m}
\pos{\lambda_i-\varphi_i(x)},
\qquad \forall x\in X.
\end{gathered}
\end{equation}

\end{definition}

The set $E$ is nonempty, since
$\varphi_i(\lambda_i u_i)=\lambda_i$
and hence $\lambda_i u_i\in E$
for every $i$. Moreover,
$E=\{F\geq0\}=\{r_E=0\}$.

The following proposition constructs
an explicit repair along an active direction
and establishes a residual error bound
for the finite union $E$.

\begin{proposition}\label{prop:error-bound}

Let $E$, $F$, $r_E$, and the data
$(\varphi_i,u_i,\lambda_i)_{i=1}^m$
be as in Definition~\ref{def:residual}.
For each $x\in X$, define the nonempty
set of active indices
\[
I(x):=\{i\in\{1,\ldots,m\}:
\varphi_i(x)-\lambda_i=F(x)\}.
\]

For any $x\in X$ and
$i_x\in I(x)$, define the
directional repair
\begin{equation}\label{eq:axial-repair}
z_x:=x+r_E(x)u_{i_x}.
\end{equation}
Then $z_x\in E$ and
\begin{equation}\label{eq:error-bound}
\dist(x,E)\leq\|x-z_x\|=r_E(x),
\qquad \forall x\in X.
\end{equation}

If, in addition, all $\varphi_i$
are Lipschitz, with
$L:=\max_{1\leq i\leq m}\Lip(\varphi_i)$,
then
\begin{equation}\label{eq:two-sided-residual}
r_E(x)\leq L\dist(x,E),
\qquad \forall x\in X.
\end{equation}
\end{proposition}

\begin{proof}
The active set is nonempty because the maximum is finite.
If \(x\in E\), then \(r_E(x)=0\) and \(z_x=x\).
Otherwise \(r_E(x)=\lambda_{i_x}-\varphi_{i_x}(x)>0\).
Superadditivity and normalization give
\(\varphi_{i_x}(z_x)\geq\varphi_{i_x}(x)+
r_E(x)\varphi_{i_x}(u_{i_x})=\lambda_{i_x}\).
Thus \(z_x\in E\); since \(\|u_{i_x}\|=1\), its displacement
has norm exactly \(r_E(x)\).
If \(L<+\infty\), the finite maximum \(F\) is \(L\)-Lipschitz.
For \(x\notin E\) and \(z\in E\),
\(r_E(x)=-F(x)\leq F(z)-F(x)\leq L\|x-z\|\).
Taking the infimum proves \eqref{eq:two-sided-residual};
on \(E\), both sides vanish.
\end{proof}

The directional error bound now yields
exact penalization for Lipschitz objectives,
together with quantitative estimates for
approximate penalized solutions.

\begin{theorem}\label{thm:exact-penalty}

Let $E$ and $r_E$ be as in
Definition~\ref{def:residual}, let
$g:X\to\R$ be $L_g$-Lipschitz,
and let $M>L_g$. Define
\begin{equation}\label{eq:penalized-objective}
J_M=g+Mr_E,\qquad v_E=\inf_E g.
\end{equation}

Then, even if $v_E=-\infty$,
\begin{equation}\label{eq:exact-values}
\inf_X J_M=v_E.
\end{equation}

If $v_E\in\R$, then
\begin{equation}\label{eq:merit-bound}
J_M(x)\geq v_E+(M-L_g)r_E(x),
\qquad \forall x\in X,
\end{equation}
and $\argmin_X J_M=\argmin_E g$.

In this case, for every $\tau\geq0$
and every $x\in X$ satisfying
$J_M(x)\leq v_E+\tau$,
\begin{equation}\label{eq:inexact-lift}
r_E(x)\leq\frac{\tau}{M-L_g},
\qquad g(z_x)\leq v_E+\tau,
\end{equation}

where $z_x$ is any directional
repair defined in \eqref{eq:axial-repair}.
\end{theorem}

\begin{proof}
Fix \(x\in X\) and any repair \(z_x\) from \eqref{eq:axial-repair}.
Proposition~\ref{prop:error-bound} gives \(z_x\in E\) and
\(\|x-z_x\|=r_E(x)\). In particular,
\(g(x)\geq g(z_x)-L_g r_E(x)\).
Assume first \(v_E\in\R\).
Since \(g(z_x)\geq v_E\), adding \(Mr_E(x)\) gives
\[
 J_M(x)\geq g(z_x)+(M-L_g)r_E(x)
          \geq v_E+(M-L_g)r_E(x),
\]
which proves \eqref{eq:merit-bound} for every \(x\).
The first inequality also controls the value at every admissible repair;
no choice of active index is needed beyond membership in \(I(x)\).

For the optimal values, \(r_E=0\) on \(E\), so
\(\inf_XJ_M\leq\inf_EJ_M=\inf_Eg=v_E\).
The merit bound gives the reverse inequality because
\(M-L_g>0\) and \(r_E\geq0\).
If \(v_E=-\infty\), restriction to \(E\) instead yields
\(\inf_XJ_M\leq\inf_EJ_M=-\infty\), proving
\eqref{eq:exact-values} in that case as well.
The value \(+\infty\) cannot occur for \(v_E\), since \(E\) is
nonempty and \(g\) is real-valued.

For the minimizers, suppose \(v_E\in\R\).
If \(x\in\argmin_XJ_M\), the equality of values and the merit
bound give
\(0\geq(M-L_g)r_E(x)\).
Thus \(r_E(x)=0\), \(x\in E\), and \(g(x)=J_M(x)=v_E\).
Conversely, if \(x\in\argmin_Eg\), then \(r_E(x)=0\) and
\(J_M(x)=v_E=\inf_XJ_M\).
This proves equality of the minimizer sets, including the case
when both sets are empty. It does not require prior attainment
of either infimum.

Finally, let \(\tau\geq0\) and \(J_M(x)\leq v_E+\tau\).
Combining this with \eqref{eq:merit-bound} gives
\((M-L_g)r_E(x)\leq\tau\), hence the residual estimate in
\eqref{eq:inexact-lift}.
For any repair, Lipschitz continuity also gives
\(g(z_x)\leq g(x)+L_g r_E(x)
=J_M(x)-(M-L_g)r_E(x)\leq v_E+\tau\).
Its feasibility implies \(g(z_x)\geq v_E\).
Thus an inexact penalized point produces, by every directional repair,
a feasible point with no larger tolerance in objective value.
\end{proof}

We now combine exact penalization
for the finite approximations $D^\eta$
with recovery of feasible points in $D$.
The resulting estimates control the recovery
distance and the associated objective-value gaps.

\begin{theorem}\label{thm:abstract-optimization}

Let $D\subset X$ be a nonempty closed,
co-radiant, and directionally compact set
with $0_X\notin D$, and let $\cF$ be a
convex directionally fine family of 
closed cones in $X$.

Fix $0<\eta<1$ and choose a finite
relative $\eta$-outer approximation
$D^\eta=\bigcup_{i=1}^{m_\eta}
\{\varphi_{i,\eta}\geq\lambda_{i,\eta}\}$
from Theorem~\ref{thm:relative-approx},
with superlinear selected generators
$\varphi_{i,\eta}$. Define its residual by
$r_\eta(x):=\min_{1\leq i\leq m_\eta}
\pos{\lambda_{i,\eta}-\varphi_{i,\eta}(x)}$,
for every $x\in X$.

Let $g:X\to\R$ be $L_g$-Lipschitz
and $M>L_g$, and set
$v:=\inf_D g$, $J_{\eta,M}:=g+Mr_\eta$,
and $v_\eta:=\inf_X J_{\eta,M}$.

Then $v_\eta=\inf_{D^\eta}g\leq v$.

If $x_\eta\in\argmin_X J_{\eta,M}$,
then $v\in\R$ and there exists
$d_\eta\in D$ such that, with
$c_\eta:=\eta/(1-\eta)$,
\begin{alignat}{2}
 &x_\eta\in D^\eta,
 &\qquad&
 \|x_\eta-d_\eta\|
 \leq c_\eta\|x_\eta\|,
 \label{eq:abstract-opt-distance}\\
 &0\leq g(d_\eta)-v
 \leq L_g c_\eta\|x_\eta\|,
 &\qquad&
 0\leq v-v_\eta
 \leq L_g c_\eta\|x_\eta\|.
 \label{eq:abstract-opt-values}
\end{alignat}
\end{theorem}

\begin{proof}
Theorem~\ref{thm:exact-penalty} gives
\(v_\eta=\inf_{D^\eta}g\leq v\).
If \(x_\eta\) is a minimizer, \(v_\eta=J_{\eta,M}(x_\eta)\) is finite;
since \(D\ne\varnothing\), this implies \(v\in\R\).
Exactness then gives \(x_\eta\in D^\eta\) and
\(g(x_\eta)=v_\eta\).
Choose \(d_\eta\) by \eqref{eq:abstract-recovery}.
Both nonnegative value gaps are bounded by
\(g(d_\eta)-g(x_\eta)\leq L_g\|d_\eta-x_\eta\|\).
The recovery estimate proves the stated bounds.
\end{proof}

In particular, if $\|x_\eta\|\leq T$
for some $T>0$, then, for every
$\varepsilon>0$, the recovered point
$d_\eta$ is $\varepsilon$-optimal for the
original problem whenever
$0<\eta\leq\varepsilon/(L_gT+\varepsilon)$
and $\eta<1$.

The following corollary extends the
recovery and value estimates to approximate
minimizers of the penalized problem.

\begin{corollary}\label{cor:abstract-inexact-optimization}

Under the assumptions and notation of
Theorem~\ref{thm:abstract-optimization},

suppose $v_\eta\in\R$.
Let $\tau\geq0$, $T\geq0$, and $x\in X$
satisfy $J_{\eta,M}(x)\leq v_\eta+\tau$
and $\|x\|\leq T$.
Set $a:=\tau/(M-L_g)$ and
$c_\eta:=\eta/(1-\eta)$.
Then $v\in\R$ and there exists
$d\in D$ such that
\begin{align}
\|x-d\|
&\leq a+c_\eta(T+a),
\label{eq:abstract-inexact-distance}\\
0\leq g(d)-v
&\leq\tau+L_g c_\eta(T+a),
\label{eq:abstract-inexact-gap}\\
0\leq v-v_\eta
&\leq\tau+L_g c_\eta(T+a).
\label{eq:abstract-inexact-value}
\end{align}
\end{corollary}

\begin{proof}
Theorem~\ref{thm:exact-penalty} supplies a repair \(z_x\in D^\eta\)
with \(\|x-z_x\|\leq a\) and \(g(z_x)\leq v_\eta+\tau\).
The triangle inequality gives \(\|z_x\|\leq T+a\).
Recover \(d\in D\) by \eqref{eq:abstract-recovery}, so
\(\|z_x-d\|\leq c_\eta(T+a)\).
Adding the two displacement bounds proves the distance estimate.
Lipschitz continuity gives
\(g(d)\leq v_\eta+\tau+L_g c_\eta(T+a)\).
Since \(v_\eta\leq v\leq g(d)\) and \(v_\eta\in\R\),
the original value is finite. Subtracting \(v\) or \(v_\eta\)
from this upper bound proves both value estimates.
\end{proof}

We now establish convergence of bounded
approximate penalized minimizers as both the
approximation parameter and the optimization
tolerance tend to zero.

\begin{theorem}\label{thm:convergence}

Under the assumptions and notation of
Theorem~\ref{thm:abstract-optimization},

let $(\eta_n)$ and $(\tau_n)$ satisfy
$0<\eta_n<1$, $\eta_n\to0$, and
$\tau_n\geq0$, $\tau_n\to0$.
For each $n$, choose a superlinear
approximation $D^{\eta_n}$ as in that theorem
and set $v_n:=\inf_X J_{\eta_n,M}$.
Suppose $v_n\in\R$ and $x_n\in X$ satisfy
$J_{\eta_n,M}(x_n)\leq v_n+\tau_n$
for every $n$, with
$\sup_n\|x_n\|<+\infty$.

Then there exist recovered points
$d_n\in D$ such that
\[
\|x_n-d_n\|\to0,\qquad
g(d_n)\to v,\qquad v_n\to v.
\]

The sequences $(x_n)$ and $(d_n)$
are relatively norm compact, and every
cluster point of either sequence minimizes
$g$ over $D$. In particular,
$\argmin_D g\neq\varnothing$.
If the minimizer is unique, both sequences
converge to it.
\end{theorem}

\begin{proof}
Put \(T=\sup_n\|x_n\|\), \(a_n=\tau_n/(M-L_g)\), and
\(c_n=\eta_n/(1-\eta_n)\).
Corollary~\ref{cor:abstract-inexact-optimization} gives \(v\in\R\)
and recovered points \(d_n\in D\).
Since \(a_n,c_n,\tau_n\to0\), its three bounds give all stated limits.
They also imply boundedness of \(d_n\).

By Proposition~\ref{prop:profile}, the sequence \(d_n\) lies
in a compact bounded section of \(D\).
Every subsequence therefore has a convergent further subsequence.
The corresponding \(x_n\) has the same limit because
\(\|x_n-d_n\|\to0\), proving relative norm compactness of both sequences.
Each limit belongs to \(D\) by closedness, and continuity gives
\(g(d)=\lim g(d_n)=v\).
Thus every cluster point is a minimizer and at least one exists.
If it is unique, relative compactness excludes any subsequence
remaining a positive distance from it, proving convergence of
both full sequences.
\end{proof}

\begin{remark}\label{rem:existence}

Suppose, in addition to the preceding
hypotheses, that $g$ is coercive, that is,
$g(x)\to+\infty$ as $\|x\|\to+\infty$.
Since $r_\eta$ is continuous and
nonnegative, each penalized objective
$J_{\eta,M}=g+Mr_\eta$ is continuous
and coercive. Thus, in finite dimensions,
every penalized problem admits a minimizer.

Coercivity also provides the uniform
boundedness required in
Theorem~\ref{thm:convergence}.
Indeed, fix $d\in D$ and let $x_\eta$
satisfy
$J_{\eta,M}(x_\eta)\leq v_\eta+\tau_\eta$.
Since $r_\eta(d)=0$, we have
\[
g(x_\eta)\leq J_{\eta,M}(x_\eta)
\leq v_\eta+\tau_\eta
\leq g(d)+\tau_\eta.
\]
If the tolerances $\tau_\eta$ are bounded,
the points $x_\eta$ lie in a common bounded
sublevel of $g$. This conclusion remains
valid in infinite dimensions, although
coercivity alone does not guarantee
attainment of the penalized minima.
\end{remark}

The sharper recovery bound of
Theorem~\ref{thm:uad-recovery} also improves
the optimization estimates for axial
approximations.

\begin{remark}\label{rem:axial-specialization}

When the approximations $D^\eta$ and
their residuals are constructed using the axial
generators of Theorem~\ref{thm:uad-recovery},
the conclusions of
Theorem~\ref{thm:abstract-optimization}
and Corollary~\ref{cor:abstract-inexact-optimization}
remain valid, under their respective hypotheses,
with $c_\eta:=\eta$ in place of
$\eta/(1-\eta)$.

Indeed, for an exact penalized minimizer
$x_\eta$, choose any axial piece containing
$x_\eta$ and set
$d_\eta:=d_i(x_\eta)$ as in
\eqref{eq:uad-recovery-map}.

For an approximate penalized minimizer
$x$, first obtain $z_x\in D^\eta$ by the
directional repair \eqref{eq:axial-repair},
and then set $d:=d_i(z_x)$ using any
axial piece containing $z_x$.

The convergence conclusions of
Theorem~\ref{thm:convergence} also hold
with these recovered points.
All these refinements follow by replacing
the abstract recovery estimate
\eqref{eq:abstract-recovery} with the
sharper axial estimate
\eqref{eq:uad-recovery-estimate}
in the preceding proofs.
\end{remark}

The following example applies Remark~\ref{rem:axial-specialization}
to a three-ray production model, using axial approximations and their
explicit recovery maps to obtain feasible points and certified value bounds.

\begin{example}
\label{ex:coupled-costs}
In Euclidean \(\R^3\), let
\[
D=\bigcup_{i=1}^3\{te_i:t\geq1\},\qquad
\psi(t)=\sqrt{1+t^2}-1,
\]
and define
\[
g(x)=\psi(x_1+2x_2-4)+2\psi(x_2+2x_3-3)+\psi(2x_1+x_3-5).
\]
The feasible set describes the choice of one of three machines operated
at level \(t\geq1\). Their three outputs are
\(x_1+2x_2\), \(x_2+2x_3\), and \(2x_1+x_3\), with targets
\(4,3,5\). The even function \(\psi\) penalizes both excess and
shortfall; the second target receives double weight.

The matrix of the three linear forms has determinant \(9\).
Since \(\psi''(t)=(1+t^2)^{-3/2}>0\) and
\(\psi(t)\geq|t|-1\), invertibility gives strict convexity and coercivity
of \(g\). Moreover, \(|\psi'|\leq1\) gives \(|\partial_1g|\leq3\),
\(|\partial_2g|\leq4\), and \(|\partial_3g|\leq5\); hence \(g\) is
\(L_g\)-Lipschitz with \(L_g=5\sqrt2\).
On the second ray, the derivative is
\(2\psi'(2t-4)+2\psi'(t-3)\), which vanishes at \(t=7/3\);
strict convexity makes this the unique minimizer.
For the other rays, use
\[
 \sqrt{1+a^2}+\sqrt{1+b^2}
 \geq\sqrt{4+(|a|+|b|)^2}.
\]
For \((a,b)=(t-4,2t-5)\), the identity \(b-2a=3\) gives
\(|a|+|b|\geq3/2\).
For \((a,b)=(2t-3,t-5)\), the identity \(a-2b=7\) gives
\(|a|+|b|\geq7/2\); also \(\sqrt{1+a^2}\geq1\).
Consequently, for \(t\geq1\),
\[
g(te_1)\geq2\sqrt{10}-\frac32,\qquad
g(te_3)\geq\sqrt{17}+\frac{\sqrt{65}}2-3.
\]
Both bounds exceed \(\sqrt{13}+\sqrt{26}-4\). Thus
\[
d_*=(0,7/3,0),\qquad
v=\min_D g=\sqrt{13}+\sqrt{26}-4\simeq4.704571.
\]

For \(0<\eta\leq1/2\), put \(\lambda=(1+\eta)^{-1}\),
\(\delta=\eta/2\), and \(Q_i x=x-x_i e_i\). Define
\[
\varphi_i(x)=x_i-\delta^{-1}\|Q_i x\|_2,\qquad
G_i^\eta=\{\varphi_i\geq\lambda\},\qquad
D^\eta=\bigcup_{i=1}^3G_i^\eta.
\]
Then \(D\subset\intt D^\eta\subset D_\eta\). Taking $M=8>L_g$, the residual and exact penalty are
\[
r_\eta(x)=\min_i\pos{\lambda-\varphi_i(x)},\qquad
J_{\eta,M}(x)=g(x)+8r_\eta(x).
\]
Let \(x_\eta\) minimize \(J_{\eta,M}\), put
\(v_\eta=J_{\eta,M}(x_\eta)=g(x_\eta)\), and choose \(i_\eta\) with
\(x_\eta\in G_{i_\eta}^\eta\). The recovery
\[
d_\eta=\max\{(x_\eta)_{i_\eta},1\}e_{i_\eta}
\]
belongs to \(D\). With \(B_\eta=5\sqrt2\,\|x_\eta-d_\eta\|_2\), one has
\[
0\leq v-v_\eta\leq B_\eta,\qquad
0\leq g(d_\eta)-v\leq B_\eta,\qquad
B_\eta\leq5\sqrt2\,\eta\|x_\eta\|_2.
\]
Coercivity and \(g(x_\eta)\leq v\) give uniform boundedness. Hence
Theorem~\ref{thm:convergence}, with the recovery of
Remark~\ref{rem:axial-specialization}, yields \(x_\eta,d_\eta\to d_*\)
and \(v_\eta,g(d_\eta)\to v\).
Table~\ref{tab:coupled-costs} compares the minima on the three convex
pieces, whose constraints are
\(\|Q_i x\|_2\leq\delta(x_i-\lambda)\).
At the two coarsest accuracies the selected
piece corresponds to the first ray; the original minimizer lies on the second.
\end{example}

\begin{table}[htbp]
\centering
\small
\setlength{\tabcolsep}{4pt}
\caption{Example~\ref{ex:coupled-costs}, rounded to six decimal places.
The bound \(B_\eta\) controls both value gaps.}
\label{tab:coupled-costs}
\begin{tabular}{rrcrrr}
\toprule
\(\eta\) & \(i_\eta\) & \(x_\eta\) & \(v_\eta\) & \(g(d_\eta)\) & \(B_\eta\)\\
\midrule
0.400 & 1 & \((2.803489,0.290933,0.299915)\) & 3.192344 & 5.053722 & 2.954579\\
0.200 & 1 & \((2.778366,0.139160,0.135890)\) & 4.142095 & 5.047816 & 1.375346\\
0.100 & 2 & \((0.030685,2.323951,0.063742)\) & 4.466063 & 4.704724 & 0.500229\\
0.050 & 2 & \((0.014363,2.331122,0.031333)\) & 4.586329 & 4.704579 & 0.243729\\
0.010 & 2 & \((0.002720,2.333317,0.006140)\) & 4.681189 & 4.704571 & 0.047490\\
0.001 & 2 & \((0.000269,2.333342,0.000611)\) & 4.702240 & 4.704571 & 0.004718\\
\bottomrule
\end{tabular}
\end{table}

\section{Relation to Earlier Work and Discussion}\label{sec:comparison}

We compare the present results with
earlier separation, polarity, and approximation
theories, emphasizing the role of directional
fineness and distinguishing the contributions
of the prescribed generating classes from
established geometric principles.

In a real normed space, Zaffaroni
\cite[Theorems~3.1 and~3.4]{Zaffaroni2007} characterizes closed radiant
sets \(R\) by the existence, for each \(y\notin R\), of a continuous
superlinear function \(p\) with \(p(y)>1\geq\sup_Rp\).
The corresponding characterization of closed co-radiant sets \(D\)
uses \(p(y)>-1\geq\sup_Dp\) for \(y\notin D\).
For convex directionally fine families, Theorem~\ref{thm:bilateral}
realizes both statements within \(\cP_{\cF}^{\nc}\), using generators
selected independently of the sets and points. The refinement is the
sufficiency of this prescribed norm-corrected class; continuous
superlinear separation in arbitrary normed spaces is not itself new.

For Banach spaces with an LUR norm, Zaffaroni
\cite[Theorem~4.2]{Zaffaroni2022} shows that both separation statements
hold with normlinear functions
\(x\mapsto a(x)+b\|x\|\), \(a\in X^*\), \(b\leq0\).
Our norm-corrected axial functions \eqref{eq:ncuad} use a norming functional,
the norm, and the additional transverse term \(\|Q_{u,f}x\|\).
They give explicit superlinear separators in every real normed space.
Thus the explicit realization no longer requires an LUR norm.
This is a change of separating class, not an extension of the normlinear
theorem to arbitrary norms within its original class.

The evaluation polarities of \cite[Section~3]{Zaffaroni2007}
already identify closed radiant and closed co-radiant sets as bipolar
fixed points. In the normlinear LUR setting, the discussion preceding
\cite[Proposition~6.1]{Zaffaroni2022} gives the primal bipolar hulls;
that proposition then establishes bijections between the corresponding
closed sets and their polars. Theorem~\ref{thm:bipolar} obtains the
same primal hulls for the present generating classes:
\(\cl(\rad E)\) for nonempty \(E\), and \(\cl(\sdw E)\) when
\(\dist(0_X,E)>0\). In our sign convention these are the positive and
negative bipolars. Corollary~\ref{cor:ncuad-bipolar} realizes both
identities with norm-corrected axial functions in arbitrary normed spaces.
These identities follow from the sufficiency of the selected separators;
the primal hulls and the fixed-point property are not new geometric
characterizations.

In the same LUR setting, \cite[Proposition~6.2]{Zaffaroni2022}
represents the gauge of a closed radiant set as a supremum over its
normlinear polar, and the cogauge of a closed co-radiant set \(D\)
as an infimum over its reverse polar on \(\cl(\cone D)\).
Here Corollary~\ref{cor:radial-levels} identifies the admissible
separating levels through the radial profiles.
Corollary~\ref{cor:strict-gauge-representations} then recovers both
functions using strict separators from \(\cP_{\cF}^{\nc}\):
\(\mu_R=\sup_{p\in\mathcal S_R^{\cF}}p\) and
\(\nu_D=-\sup_{p\in\mathcal S_D^{\cF}}p\), both on all of \(X\).
These representations hold in every real normed space and apply
in particular to the norm-corrected axial family. Their contribution is the
representation by the present strict-separator classes, rather than a
new notion of gauge or cogauge.

In our earlier work on nonlinear separation
of co-radiant sets
\cite{GarciaMelguizoJOGO2025},
we studied separation and scalarization
for approximate and proper approximate
efficiency, using norm-bases and SSP conditions
in different separation results.
The present approximation theory relies instead
on compactness of the direction set $\Sigma_D$.
For closed co-radiant sets avoiding the origin,
the existence of a norm-base amounts to
$\mathcal I_D<+\infty$, whereas directional
compactness does not require this bound.
These conditions are not equivalent in general.
Moreover, the present results concern finite
relative outer approximations and quantitative
recovery, rather than optimality
characterizations for approximate efficiency.

Compact minimal bases also play an important
role in the separation and approximate-efficiency
results of
\cite{SayadiBanderEtAl2017}.
For a closed co-radiant set, the radial
base $\Gamma_D$ is its minimal base under
dilation. By Proposition~\ref{prop:profile},
compactness of $\Gamma_D$ is equivalent to
compactness of $\Sigma_D$ together with
continuity of $\rho_D$ on $\Sigma_D$.
The finite approximation results developed here
require only compactness of $\Sigma_D$,
without assuming continuity or boundedness
of the radial profile.

Directional fineness, adapted generators, and the uniform axial
deviation family were developed in our earlier cone paper
\cite[Definition~2.6, Proposition~2.9, Corollary~5.6]
{GarciaMelguizo2026}. That paper also proves finite outer
approximation and separation for locally compact cones
\cite[Theorems~3.21 and~3.23]{GarciaMelguizo2026}, and obtains
axial penalization and recovery for conically constrained problems
\cite[Theorem~5.8 and Corollary~5.9]{GarciaMelguizo2026}.
These constructions are used here as the geometric starting point.

The present results concern closed radiant and co-radiant sets.
For a co-radiant set, the radial threshold $\rho_D(u)$ varies with
the direction and need be neither continuous nor bounded on the
compact set $\Sigma_D$. Accordingly, the finite approximants are
unions of positive superlevels
$\{\varphi_i\geq\lambda_i\}$, rather than unions of cones.
Their residual controls distance to $D^\eta$; a separate recovery
step controls distance to $D$. Bilateral norm-corrected separation,
radial calibration, the prescribed-class bipolar and gauge
representations, and the relative recovery estimates address this
additional radial geometry. In particular, the factor $\eta$ in
the axial recovery theorem improves the general bound in this
paper; axial recovery with that factor was already available for
cones in \cite[Theorem~5.8]{GarciaMelguizo2026}.

\section{Conclusions}\label{sec:conclusion}

The separation, bipolar, and gauge representations realize known
geometric principles within classes prescribed by a directionally
fine family. When convex localized witnesses are selected, the
corresponding separators are Lipschitz and superlinear. The axial
family supplies such witnesses explicitly in every real normed
space.

It dispenses with the LUR assumption by changing
the separating class, not by extending normlinear
separation within the original class.

The main quantitative contribution is
finite relative outer approximation of closed
co-radiant sets with compact direction set,
together with separation preservation and
directional and bounded-section error estimates.
No continuity or boundedness of the radial
profile is required.

A directional repair gives a residual
error bound for each finite approximation,
yielding exact penalization for Lipschitz
objectives.

Recovery in the original set controls
distance and objective-value errors separately
from the optimization tolerance. As the
approximation parameter and optimization
tolerance vanish, bounded approximate penalized
minimizers are relatively norm compact, and
every cluster point minimizes the original
objective. Convergence of the full sequence
follows when the minimizer is unique.
The explicit axial map improves the relative
distance constant from $\eta/(1-\eta)$ to $\eta$.

The axial family gives explicit local cones,
generators, and recovery maps once norming
functionals are chosen; the parameter estimates
guide their construction.

Compactness guarantees the existence
of a finite global selection, but provides
neither a constructive selection procedure
nor an a priori bound on the number of pieces.
Obtaining such bounds may require quantitative
information on directional coverings, radial
regularity, and the generators; certified
procedures would also require effective access
to the geometric data.
Intrinsic descriptions of functional polars
and extensions beyond global Lipschitz objectives
remain further questions.

\section*{Funding}

This research was supported by projects
PID2021-122126NB-C32, funded by
MICIU/AEI/10.13039/501100011033 and by FEDER,
``A way of making Europe'';
and PID2025-168246NB-I00, funded by
MICIU/AEI/10.13039/501100011033 and by ERDF/EU.


\begin{thebibliography}{99}

\bibitem{GarciaMelguizo2026}
F. García-Castaño and M. Á. Melguizo-Padial,
Directionally Fine Families of Cones:
Construction Principles, Approximation, and Separation,
arXiv preprint arXiv:2609.15337 (2026).


\bibitem{Zaffaroni2004}
A. Zaffaroni,
Is every radiant function the sum of quasiconvex functions?,
Math. Methods Oper. Res. 59 (2004), 221--233.

\bibitem{Zaffaroni2006}
A. Zaffaroni,
Superlinear separation and dual properties of radiant functions,
Pac. J. Optim. 2 (2006), 181--202.

\bibitem{Zaffaroni2007}
A. Zaffaroni,
Superlinear separation for radiant and coradiant sets,
Optimization 56 (2007), 267--285.

\bibitem{Zaffaroni2008}
A. Zaffaroni,
Convex coradiant sets with a continuous concave cogauge,
J. Convex Anal. 15 (2008), 325--343.

\bibitem{Mohebi2013}
H. Mohebi,
Abstract convexity of radiant functions with applications,
J. Global Optim. 55 (2013), 521--538.

\bibitem{Zaffaroni2013}
A. Zaffaroni,
Convex radiant costarshaped sets and the least sublinear gauge,
J. Convex Anal. 20 (2013), 307--328.

\bibitem{HansenEtAl2020}
G. Hansen, I. Herburt, H. Martini, M. Moszy\'nska,
Starshaped sets,
Aequationes Math. 94 (2020), 1001--1092.

\bibitem{GutierrezJimenezNovo2006MMOR}
C. Guti\'errez, B. Jim\'enez, V. Novo,
On approximate efficiency in multiobjective programming,
Math. Methods Oper. Res. 64 (2006), 165--185.

\bibitem{GutierrezJimenezNovo2006SIAM}
C. Guti\'errez, B. Jim\'enez, V. Novo,
A unified approach and optimality conditions for approximate solutions
of vector optimization problems,
SIAM J. Optim. 17 (2006), 688--710.


\bibitem{Qiu2013}
J.-H. Qiu,
Set-valued quasi-metrics and a general Ekeland's variational principle
in vector optimization,
SIAM J. Control Optim. 51 (2013), 1350--1371.

\bibitem{SayadiBanderEtAl2017}
A. Sayadi-Bander, L. Pourkarimi, R. Kasimbeyli, H. Basirzadeh,
Coradiant sets and $\varepsilon$-efficiency in multiobjective optimization,
J. Global Optim. 68 (2017), 587--600.

\bibitem{SayadiBanderKasimbeyliPourkarimi2017}
A. Sayadi-Bander, R. Kasimbeyli, L. Pourkarimi,
A coradiant based scalarization to characterize approximate solutions
of vector optimization problems with variable ordering structures,
Oper. Res. Lett. 45 (2017), 93--97.


\bibitem{PengWeiKasimbeyli2025}
J.-W. Peng, W.-B. Wei, R. Kasimbeyli,
Linear and nonlinear scalarization methods for vector optimization
problems with variable ordering structures,
J. Optim. Theory Appl. 206 (2025), Article 2.


\bibitem{AnsariSharma2022}
Q.H. Ansari, P.K. Sharma,
Some properties of generalized oriented distance function and their
applications to set optimization problems,
J. Optim. Theory Appl. 193 (2022), 247--279.


\bibitem{MartinezLegazRubinovSchaible2005}
J.E. Mart\'inez-Legaz, A.M. Rubinov, S. Schaible,
Increasing quasiconcave co-radiant functions with applications
in mathematical economics,
Math. Methods Oper. Res. 61 (2005), 261--280.

\bibitem{Zaffaroni2022}
A. Zaffaroni,
Separation, convexity and polarity in the space of normlinear functions,
Optimization 71 (2022), 1213--1248.

\bibitem{GRTZ2023}
A. G\"opfert, H. Riahi, C. Tammer, C. Z\u{a}linescu,
Variational Methods in Partially Ordered Spaces, 2nd ed.,
Springer, Cham, 2023.


\bibitem{GarciaMelguizoJOGO2025}
F. García-Castaño, M. Á. Melguizo-Padial,
Nonlinear separation theorems for co-radiant sets
and optimality conditions for approximate and
proper approximate solutions in vector optimization,
J. Global Optim. 93 (2025), 1079--1095.
https://doi.org/10.1007/s10898-025-01561-x.

\end{thebibliography}
\end{document}